\documentclass{article}

\usepackage{amsfonts}
\usepackage{makeidx}
\usepackage{amsmath}
\usepackage{amssymb}
\usepackage{amsthm}
\usepackage[final]{graphicx}
\usepackage[all]{xy}
\usepackage{float}
\usepackage{fancybox}
\usepackage{color}

\usepackage{multicol}
\usepackage{hyperref}
\usepackage{cite}

\newcommand{\G}{\mathcal{G}}
\newcommand{\Spec}{\mathrm{Spec}}

\newcommand{\ord}{\mathrm{ord}}

\newcommand{\Mm}{\mathrm{\underline{Max}\; mult}_X}
\newcommand{\mm}{\mathrm{max\, mult}_X}
\newcommand{\Diff}{\mathrm{Diff}}
\newcommand{\Sing}{\mathrm{Sing}}

\DeclareMathOperator{\mult}{mult}

\newcommand{\R}{\mathcal{R}}

\newcommand{\Gn}{\G^{(n)}}

\newcommand{\Gne}{\G^{(n-e)}}
\newcommand{\Gd}{\G^{(d)}}

\newcommand{\Ovne}{\mathcal O_{V^{(n-e)}}}  
\newcommand{\Vn}{V^{(n)}}
\newcommand{\Vne}{V^{(n-e)}}
\newcommand{\Vd}{V^{(d)}}
\newcommand{\Ovd}{\mathcal O_{V^{(d)}}}  
\usepackage{cancel} 
\usepackage[normalem]{ulem}

\theoremstyle{plain}
\newtheorem{Thm}{Theorem}[section]         %Numeración dentro de secciones

\newtheorem{Cor}[Thm]{Corollary}
   
\newtheorem{Prop}[Thm]{Proposition}

\theoremstyle{definition}
\newtheorem{Parrafo}[Thm]{\ }
\newtheorem{Def}[Thm]{Definition}  
\newtheorem{Def-Prop}[Thm]{Definition-Proposition}

\theoremstyle{remark}
\newtheorem{Rem}[Thm]{Remark}

\newtheorem{Ex}[Thm]{Example}%[chapter]

\numberwithin{equation}{Thm} %Numeración dentro de teoremas

\begin{document}

\title{Nash multiplicity sequences and Hironaka's order function}
\author{A. Bravo, S. Encinas, B. Pascual-Escudero
\thanks{The authors were partially supported by MTM2015-68524-P. The third author was supported by BES-2013-062656.}}
%\footnote{The authors were partially supported by MTM2015-68524-P. The third author was supported by BES-2013-062656.\\
%{\em Mathematics subject classification. 14E15, 14E18.}\\
%\textit{Keywords:} Rees algebras. Resolution of Singularities. Arc Spaces}} 
%
\begingroup%Locallizing the change to `thefootnote'.
    \renewcommand{\thefootnote}{}%Removing the footnote symbol.
\endgroup

\AtEndDocument{\bigskip{\footnotesize
  \textsc{Depto. Matem\'aticas,
Facultad de Ciencias, Universidad Aut\'onoma de Madrid 
and Instituto de Ciencias Matemáticas CSIC-UAM-UC3M-UCM, Canto Blanco 28049 Madrid, Spain} \par  
  \textit{E-mail address}, A. Bravo: \texttt{ana.bravo@uam.es} \par
  \textit{E-mail address}, B. Pascual-Escudero: \texttt{beatriz.pascual@uam.es} \par
  \addvspace{\medskipamount}
  \textsc{Depto. Matem\'atica Aplicada,
and IMUVA, Instituto de Matem\'aticas.
Universidad de Valladolid.
} \par
  \textit{E-mail address}, S. Encinas: \texttt{sencinas@maf.uva.es}
}}

%\date{\today }
\maketitle
\normalsize

\begin{abstract}
When $X$ is a  $d$-dimensional variety defined over a  field  $k$ of characteristic zero,  a constructive resolution   of singularities   can be achieved by successively lowering  the maximum multiplicity    via blow ups at smooth equimultiple centers.  This is done by stratifying the maximum multiplicity locus of $X$  by means of the so called {\em resolution functions}. The most important of these functions is what we know as {\em Hironaka's order function in dimension $d$}.   Actually, this function can   be   defined for    varieties   when the base field is  perfect; however if the characteristic of $k$ is positive, the function is, in general,  too coarse  and does not provide enough information so as to define a resolution. It is very natural to ask what  the meaning of this function is in this case, and to try to find refinements that could lead, ultimately, to a resolution. In this paper we show that   Hironaka's order function in dimension $d$ can be read in terms of the {\em Nash multiplicity sequences}  introduced by  Lejeune-Jalabert. Therefore, the function is intrinsic to the variety and has a geometrical meaning in terms of its space of arcs. 
\end{abstract} 

\section*{Introduction}

After Hironaka's paper on resolution of singularities (\cite{Hir}),  
the work of J. Nash on the theory of arcs on an algebraic variety  $X$ was in part  motivated by the question of how much of a resolution of singularities of $X$ is intrinsic to the variety itself (\cite{Nash}). In general, a resolution of singularities of a variety is not unique, yet one may be able to identify elements in the space of arcs of $X$ that give some indication on its desingularization. This paper is motivated by this question in the context of algorithmic resolution of singularities. 

\

Let $X$ be an algebraic variety defined over a field of characteristic zero. An algorithmic resolution of singularities of $X$ consists on a procedure to construct a sequence of blow ups at regular centers, 
$$X=X_0\leftarrow X_1 \leftarrow \ldots \leftarrow X_m$$
so that $X_m$ is non singular (see \cite{V1}, \cite{V2}, \cite{B-M}, \cite{E_V97}, \cite{E_V_}, \cite{E_Hau}).  To define such a sequence one needs to stratify the points of $X$ according to the complexity of the singularities. This is done by means of what we know as {\em resolution invariants}. The first measure of the singularity at a given point $\xi\in X$ can be, for example, the multiplicity (see \cite{V}). As it turns out, this number is too coarse  and needs to be refined. Thus more invariants have to be defined:   the next invariant at $\xi\in X$ is known as {\em Hironaka's order function at $\xi$  in dimension $d$}, where $d$ is the dimension of $X$. This is a rational number obtained after describing the multiplicity stratum through $\xi$   by a  set of     equations with weights via some (local) embedding in a smooth $V$ in a neighborhood of $\xi\in X$. We denote it by $\text{ord}_{\xi}^{(d)}(X)$. All other invariants involved in resolution derive from this one (see \cite{E_V97}, \cite{Br_V2}). 

\ 

In \cite{Br_E_P-E} we showed that $\text{ord}_{\xi}^{(d)}(X)$  can be obtained by using the information provided by the arcs on $X$ with center  $\xi$, or more precisely, it can be read from  the so called {\em Nash multiplicity sequences} of arcs with center $\xi$, introduced by Lejeune-Jalabert in \cite{L-J}. Therefore, this number used in algorithmic resolution is indeed intrinsic to $X$. Moreover it       has a geometrical meaning in terms of the  arcs of $X$ with center $\xi$ and the rate at which their graphs {\em separate from the stratum of points with the same multiplicity as $\xi$}.    See Example~\ref{ejemplo_car}.  

\

We    do not know whether there is a  theorem of resolution for varieties  defined over a field of positive characteristic (there are only  positive answers for dimension less than or equal to three, see \cite{Abhy1}, \cite{Abhy2}, \cite{Benito_V}, \cite{CJS_prep}, \cite{Cos_Pilt1}, \cite{Cos_Pilt2},   \cite{MR3617780}, \cite{Lipman3}). However,  it is still  possible to  define Hironaka's order function in any  dimension $d$ at a singular point $\xi\in X$ whenever $X$ is defined over a perfect field. It is very natural to ask what the meaning of this invariant is in this case. 

\

In this manuscript we give a (characteristic free) proof of the fact that this invariant can be read in terms of the Nash multiplicity sequence of arcs with center $\xi\in X$,  extending the results in \cite{Br_E_P-E},  and   giving an interpretation of the meaning of this number in any characteristic.  The strategy followed in the present  paper differs from the one in \cite{Br_E_P-E}, where we strongly used the characteristic zero hypothesis on the base field.  

\ 

In the following paragraphs we give more details on how Hironaka's order function is defined and how the Nash multiplicity sequence of an arc is constructed. 

\

{\bf Arc spaces, singularities, and Nash multiplicity sequences} 

\

The spaces of arcs and jets of an algebraic variety $X$ often encode information about its singularities, and during the last  few decades, they  have been widely studied by several authors (see for instance   \cite{deF_Doc},  \cite{D_L}, \cite{E_M04},   \cite{E_M09},   \cite{E_M_Y}, \cite{I}, \cite{I_K},  \cite{L-J_Mou_Re}, \cite{L-J_Re1},   \cite{Mou4},  \cite{Mus1}, \cite{Mus2}     or \cite{Veys}  among many others).

\

It is in this context of arc spaces where the \textit{Nash multiplicity sequence} appears. It was defined by M. Lejeune-Jalabert \cite{L-J} as a non-increasing sequence of positive integers attached to a germ of a curve inside a germ of a hypersurface.  This sequence of numbers can somehow be interpreted as a refinement of the multiplicity of the hypersurface at a given point: it can be seen as the {\em multiplicity along a given arc}. 

\

 M. Hickel generalized this notion  in  \cite{Hickel05}  by defining  a sequence of blow ups  that allows us to compute Nash multiplicity sequences and study their behaviour for arbitrary varieties. Given a variety $X$ defined over a field $k$, fix an arc $\varphi$ with center a (non-necessarily closed) point   $\xi$ of multiplicity $m$ (which we may assume to be the maximum multiplicity at points of $X$). Now   $\varphi$ naturaly induces another arc  $\Gamma_0$ on $X_0=X\times \mathbb{A}^1_k$ related to its  graph. Then one  can define   a sequence of blow ups at points:
\begin{equation}
\label{dirigido}
\xymatrix@R=0pt@C=30pt{\text{Spec}(K[|t|]) \ar[dddd]^{\Gamma_0} \ar[ddddr]^{\Gamma_1} \ar[ddddrrr]^{\Gamma_r} \\
	   \\   \\  \\ 
	X_0=X\times \mathbb{A}^1_k & X_1 \ar[l]  & \ldots \ar[l] & X_r \ar[l]\mbox{,}\\
	\xi _0=(\xi ,0) & \xi _1 & \ldots & \xi _r 
}\end{equation}

where $\xi _i$ is the center of the arc  $\Gamma_i$,   the lifting  of   $\Gamma_{i-1}$  to $X_i$, for $i=1, \ldots, r$, and $K$ is some field containing $k$. 
 The Nash multiplicity sequence of $\varphi$  is then the sequence
\begin{equation}
	\label{Nash_sequence_3}
	m=m_0\geq m_1\geq \ldots \geq m_r\geq 1\mbox{,}
	\end{equation}
in which $m_i$ is the multiplicity of $X_i$ at $\xi _i$ for $i=0,\ldots ,r$ (see section \ref{RAandNash} for details). We will refer to diagram (\ref{dirigido}) as {\em  the sequence of blow ups directed by $\varphi$}.  

\

In this paper we will be interested in the number of blow ups needed until  the Nash multiplicity drops below $m$ for the first time. This number will be finite whenever the generic point of  $\varphi$ is not contained in the stratum of (maximum) multiplicity $m$ of $X$, $\Mm$.  We will call this the {\em persistance of $\varphi$ in $X$}  and will denote it by $\rho _{X,\varphi }$. In other words,     $\rho _{X,\varphi }$ is such that $m=m_0=\ldots =m_{\rho _{X,\varphi }-1}>m_{\rho _{X,\varphi }}$ in the sequence (\ref{Nash_sequence_3}) above.

\

 We will also define a refinement of  $\rho _{X,\varphi }$, the \textit{order of contact} of $\varphi $ with $\Mm$,   and denote it by 
$r_{X,\varphi }$. This is a rational number whose integral part is $\rho _{X,\varphi }$ (see Proposition \ref{thm:r_elim_amalgama}). Normalizing $r_{X,\varphi }$ by the order of the arc (see Definition \ref{arco_def}) we obtain:
\begin{equation} \label{formula_ri}
\bar{r}_{X,\varphi }:=\frac{r_{X,\varphi }}{\nu_t(\varphi )}\in \mathbb{Q}_{\geq 1}\mbox{,}
\end{equation}
and 
\begin{equation}\label{eq:set_ri}
\Phi _{X,\xi }=\left\{ \overline{r}_{X,\varphi } \right\} _{\varphi }\subset \mathbb{Q}_{\geq 1}\mbox{,}
\end{equation}
where $\varphi $ runs over all arcs in $X$ with center $\xi $. Note that the set 
$\Phi _{X,\xi }$ is an invariant of $X$ at $\xi$. As we will see, the infimum (actually the minimum) of this set is related to Hironaka's order function.

\

{\bf Algorithmic resolution, local presentations, and Hironaka's order function}

\

Let $X$ be an algebraic variety defined over a perfect field $k$. One way to approach an algorithmic resolution of singularities of an algebraic variety $X$ is by classifying its singular points according to their complexity. As a first step one can consider the multiplicity at each point  of $X$ (recall that an irreducible algebraic variety is regular if and only if the multiplicity at each point equals one).   This defines an upper semicontinuous function: 
$$\begin{array}{rrcl}
\text{mult}_X:  & X & \longrightarrow & {\mathbb N} \\
& \xi & \mapsto & \text{mult}_{{\mathfrak m}_{\xi}} {\mathcal O}_{X,\xi}.
\end{array}$$
In what follows, we will denote by $\text{max mult}_X$ the maximum value of $\text{mult}_X$, and by $ \text{\underline{Max} mult}_X$ the closed set of points in $X$ where this maximum is achieved. The multiplicity function has the following nice property: if $Y\subset \text{\underline{Max} mult}_X$ is a regular center, then after blowing up at $Y$, $X\leftarrow X_1$,  one has that  $\text{max mult}_{X}\geq  \text{max mult}_{X_1}$ (see \cite{Dade}). Thus one could try to approach a resolution of singularities of $X$ by  finding a finite sequence of blow ups 
\begin{equation}
\label{bajada}
X=X_0 \leftarrow X_1\leftarrow \ldots \leftarrow X_{m-1} \leftarrow X_m
\end{equation} 
at suitable equimultiple centers $Y_i\subset \text{\underline{Max} mult}_{X_i}$
so that 
\begin{equation}
\label{bajada_m}
\text{max mult}_{X_0}=\text{max mult}_{X_1}=\ldots =\text{max mult}_{X_{m-1}}>\text{max mult}_{X_m}.
\end{equation}

\

A sequence like (\ref{bajada}) is said to be a {\em simplification of the multiplicity of $X$}. Iterating this procedure one achieves the case in which $\max\mult_{X_N}=1$ for some $N$, and this is equivalent to    $X_N$ being    regular.  
In general, for a given $X$, $\text{\underline{Max} mult}_X$ is not regular, so, in order to define the centers $Y_i$  in (\ref{bajada}) one is forced to refine the multiplicity by considering additional information on $X$. This is usually done by considering {\em local presentations for the multiplicity}. 

\

Roughly speaking, a local presentation  for the multiplicity consists of a local (\'etale) embedding of $X$ into a smooth variety $V$, in a neighborhood of $\xi\in\text{\underline{Max} mult}_X$,  together with a finite set of {\em weighted equations} whose set of zeroes coincides with 
$\text{\underline{Max} mult}_X$, and so that this description is {\em stable by blow ups at regular equimultiple centers}, at least if the maximum multiplicity of the transforms of $X$ remains constant. 

\

To clarify this statement a bit, we can think of the case where $X\subset V $ is locally a hypersurface defined by some element $f\in {\mathcal O}_V$. Then the multiplicity of $X$ at a point $\xi$ (say $m$) is given by the usual order of $f$ at the regular local ring ${\mathcal O}_{V,\xi}$, and therefore, at least locally:
$$\text{\underline{Max} mult}_X=\{\eta\in X: \text{ord}_{\eta} f\geq  m\}.$$

In \cite{V} it is shown that if $X$ is an arbitrary variety of dimension $d$ defined over a perfect field then locally, in an (\'etale) neighborhood of $\xi\in\text{\underline{Max} mult}_X$, there is an embedding in a smooth scheme $V$, elements $f_1,\ldots, f_r\in {\mathcal O}_V$ and positive integers 
$n_1, \ldots, n_r$ so that: 
\begin{enumerate}
	\item[(i)]   The subset $\text{\underline{Max} mult}_X$ can be expressed in terms of the hypersurfaces defined by the $f_i$:
	\begin{equation}
	\label{presentacion_local_1}
	\text{\underline{Max} mult}_X=\cap_{i=1}^r \{\eta\in V: \text{ord}_{\eta}f_i\geq n_i\}, 
	\end{equation}
	where $n_i$ is the maximum multiplicity of $f_i$ for $i=1,\ldots, r$; 
	\item[(ii)] The previous description is stable under blow ups at regular centers $Y\subset \text{\underline{Max} mult}_X$, i.e., if $V\leftarrow V_1$ is the blow up at $Y$, $X_1$ is the strict transform of $X$ and $f_{i,1}$ denotes the strict transform of $f_i$ in $V_i$, then $\text{max mult}_{X_1}=\text{max mult}_{X}$ if and only if 
	$$\cap_{i=1}^r \{\eta\in V: \text{ord}_{\eta}f_{i,1}\geq n_i\}\neq \emptyset$$ 
	and in this case:
	$$\text{\underline{Max} mult}_{X_1}=\cap_{i=1}^r \{\eta\in V: \text{ord}_{\eta}f_{i,1}\geq n_i\}.$$
\end{enumerate}

The embedding $X\subset V$ together with the expression (\ref{presentacion_local_1}) is what we call a {\em local presentation for the multiplicity} (see section \ref{Presentaciones_locales} for a more precise definition of what a local presentation is).  

\

Rees algebras turn out to be a convenient  tool to codify the information in a local presentation (equations and weights). It is in terms of  Rees algebras that {\em Hironaka's order function in dimension $d$} is defined, $\ord^{(d)}(X)$. This is the most important invariant in constructive resolution of singularities in characteristic zero.

\

When the characteristic of the base field is zero, it can be shown that, in fact, one can find a suitable (finite) projection to a smooth $d$-dimensional space $V'$, say $X\to V'$, and a collection of equations and weights on $V'$ that also give a local presentation of (a homeomorphic image of) the maximum multiplicity locus of $X$ 
(see \ref{EliminationProperties}). This means that $\Mm$ can be {\em represented in dimension $d$},  and this is done via a conveniently defined ${\mathcal O}_{V'}$-Rees algebra: {\em the elimination algebra} (\ref{eliminacion}). The key point is that the local presentation is stable after transformations (\ref{eliminacion_suclocal} (3)).

\

When the characteristic of the base field is positive, a finite projection as before, $X\to V'$, can be defined, and it is also possible to give a collection of equations and weights that somehow {\em approximate}  (a homemorphic image of) $\Mm$ in $V'$, again via a conveniently defined ${\mathcal O}_{V'}$-Rees algebra which we also  refer to as {\em the elimination algebra}. Therefore, we can also define Hironaka's order function in dimension $d$, $\ord^{(d)}(X)$. However,  in this context this invariant is too coarse  and does not provide enough information  to define a simplification of the multiplicity of $X$. In particular, in this case the local presentation is not stable after transformations (\ref{Inclusion_Estricta}). It is very natural to ask what the meaning of Hironaka's order function is in this case. In addition it would be very interesting to find new invariants that help refining $\ord^{(d)}(X)$. 

\

{\bf About the results in this paper}

\

The contents of this paper are motivated by  the previous question. 
In \cite{Br_E_P-E} we showed that, when the characteristic is zero,   $\ord^{(d)}_{\xi}(X)$  can be read by means of the Nash multiplicity sequence of arcs through the point $\xi\in X$. There, we strongly used  the hypothesis on the characteristic, since Tschirnhausen transformations played  a key  role in our arguments, the reason being that the  elimination algebra in this case can be constructed using  the coeficients of the elements $f_i$  (see (\ref{presentacion_local_1})) after a Tschirnhausen transformation (in a suitable \'etale neighborhood).  Here we give a unified proof of the same result over arbitrary perfect fields using the fact that in arbitrary characteristic there is a strong link between the elements $f_i$ and the elimination algebra (see \ref{eliminacion_finita}). This is the content of Theorem \ref{Main_Theorem}:

\

{\bf Theorem \ref{Main_Theorem}.}  {\em Let $X$ be an algebraic variety of dimension $d$ defined over a perfect field $k$, and let $\xi $ be a point in $\Mm $. Consider the set  $ \Phi _{X,\xi }$   defined in (\ref{eq:set_ri}). Then:
	$$\mathrm{inf} \Phi _{X,\xi }=\mathrm{min} \Phi _{X,\xi }=\ord _{\xi }^{(d)}(X)\mbox{.}$$}

Thus, it follows that $\ord _{\xi }^{(d)}(X)$    is intrinsic to $X$ and it can be read from the arcs in $X$ centered at $\xi$. In fact, it can be read from the persistance of  some   arc in $X$ (see (\ref{limite_ord_d})).  Moreover, the Theorem indicates  that it somehow measures how long it takes  at least  for an arc $\Gamma _0$  arising from $\varphi $ as explained before  to leave the maximum multiplicity stratum of $X\times \mathbb{A}^1_k$ after a suitable sequence of blow ups as in (\ref{dirigido}), giving this way a geometrical  meaning to Hironaka's order function in dimension $d$ in any characteristic. See Example \ref{ejemplo_car}.

\

{\bf How the paper is organized}  

\

In section \ref{Rees_Algebras} we recall the basics on Rees algebras when we use them as a tool in constructive resolution of singularities. As we will see, Rees algebras provide a convenient language when it comes to handling local presentations for the multiplicity, which is the content of section \ref{Presentaciones_locales}. Section \ref{ElimAlg}  is dedicated to \textit{elimination}: given a $d$-dimensional variety $X$ defined over a perfect field, a local presentation of $\Mm$ can be given by means of an embedding in a smooth scheme $V$, and a collecction of a finite set of equations with weights in $V$. However, in may situations, it is possible to give a local presentation of a homeomorphic image of $\Mm$ in some smooth $d$-dimensional  scheme. This can be done using the theory of elimination. Jets and arcs   are introduced in section 
\ref{Jet_Arcs}, while the notion of Nash multiplicity sequence, the persistance and the order of contact are given in section \ref
{RAandNash}. Finally, Theorem \ref{Main_Theorem} is proven in section \ref{Nash_Hironaka}. 

\vspace{0.2cm}

{\em Acknowledgements.} We profited from conversations with C. Abad, A. 
Benito and O. E. Villamayor.

\section{Rees algebras}\label{Rees_Algebras}

The stratum  defined by the  maximum value of the multiplicity function of a variety can be encoded using equations and weights.
The same occurs with the Hilbert-Samuel function.
Rees algebras are natural objects to work with this setting, with the advantage that we can  perform algebraic operations on them such as taking the  integral  closure or the saturation by the action of  differential operators.   

\begin{Def}
Let $R$ be a Noetherian ring. A \textit{Rees algebra $\mathcal{G}$ over $R$} is a finitely generated graded $R$-algebra
$$\mathcal{G}=\bigoplus _{l\in \mathbb{N}}I_{l}W^l\subset R[W]$$
for some ideals $I_l\in R$, $l\in \mathbb{N}$ such that $I_0=R$ and $I_lI_j\subset I_{l+j}\mbox{,\; } \forall l,j\in \mathbb{N}$. Here, $W$ is just a variable in charge of the degree of the ideals $I_l$. Since $\mathcal{G}$ is finitely generated, there exist some $f_1,\ldots ,f_r\in R$ and positive integers (weights) $n_1,\ldots ,n_r\in \mathbb{N}$ such that
\begin{equation}\label{def:Rees_alg_generadores}
\mathcal{G}=R[f_1W^{n_1},\ldots ,f_rW^{n_r}]\mbox{.}
\end{equation}
\end{Def}

Note that this definition is more general than the (usual) one considering only algebras of the form $R[IW]$ for some ideal $I\subset R$, which we call Rees rings, where all generators have weight one.

\vspace{0.4cm}

Rees algebras can be defined over Noetherian schemes in the obvious manner.

\vspace{0.2cm}

\textbf{Notation:} Let $\mathcal{G}_1,\mathcal{G}_2\subset R[W]$ be two Rees algebras. We denote by $\mathcal{G}_1\odot \mathcal{G}_2$ the smallest Rees algebra containing both. If $\mathcal{G}_1'\subset R_1[W]$, $\mathcal{G}_2'\subset R_2[W]$ for two different rings $R_1, R_2$, by abuse of notation we will sometimes denote by $\mathcal{G}'_1\odot \mathcal{G}_2'$ the Rees algebra $\mathcal{G}_1\odot \mathcal{G}_2$, where $\mathcal{G}_i$, for $i=1,2$, is the extension of $\mathcal{G}_i'$ to a Rees algebra over some  ring $R$ containing both $R_1$ and $R_2$. 

\vspace{0.2cm}

\begin{Parrafo}{\bf Notation and assumptions.} In what follows,  we will assume $k$ to be a perfect field. In general, $R$ will be a smooth $k$-algebra, and $V$ will be a smooth scheme over $k$, unless otherwise specified. We will often work locally: for many computations, we will assume that we fix a point and an open subset of $V$ containing it, so that we can reduce  to the affine case, $V=\mathrm{Spec}(R)$.
\end{Parrafo}
\vspace{0.2cm}

One can attach to a Rees algebra a closed set as follows:

\begin{Parrafo}
\textbf{The Singular Locus of a Rees Algebra.} (\cite[Proposition 1.4]{E_V}).
Let $\mathcal{G}$ be a Rees algebra over $V$. The \textit{singular locus} of $\mathcal{G}$, Sing$(\mathcal{G})$, is the closed set given by all the points $\xi \in V$ such that $\nu _{\xi }(I_l)\geq l$, $\forall l\in \mathbb{N}$, where $\nu _{\xi}(I)$ denotes the order of the ideal $I$ in the regular local ring $\mathcal{O}_{V,\xi }$. If $\mathcal{G}=R[f_1W^{n_1},\ldots ,f_rW^{n_r}]$, the singular locus of $\mathcal{G}$ can be computed as
$$\mathrm{Sing}(\mathcal{G})=\left\{ \xi \in \mathrm{Spec}(R):\, \nu _{\xi }(f_i)\geq n_i,\; \forall i=1,\ldots ,r\right\} \subset V\mbox{.}$$
\end{Parrafo}

Note that the singular locus of the Rees algebra on $V$ generated by $f_1W^{n_1}, \ldots ,f_rW^{n_r}$ does not coincide with the usual definition of the singular locus of the subvariety of $V$ defined by $f_1,\ldots ,f_r$.

\begin{Ex} \label{Ex:HiperRees}
Let $X\subset\Spec(R)=V$ be a hypersurface with $I(X)=(f)$ and let $b>1$ be the maximum value of the multiplicity of $X$.
If we set $\mathcal{G}=R[fW^b]$ then $\Sing(\mathcal{G})=\Mm$ is the set of points of $X$ having maximum multiplicity (see \ref{Def:LocPres} and  Theorem \ref{Th:PresFinita}   for a generalization of this description in the case where $X$ is an arbitrary algebraic variety with maximum multiplicity greater than 1). 
\end{Ex}

\begin{Rem}
Let $\mathcal{G}_1$ and $\mathcal{G}_2$ be two Rees algebras over $V$, then
$$\mathrm{Sing}\left( \mathcal{G}_1\odot \mathcal{G}_2 \right) =\mathrm{Sing}(\mathcal{G}_1)\cap \mathrm{Sing}(\mathcal{G}_2)\subset V\mbox{.}$$
\end{Rem}

\begin{Def}\label{def:transf_law}
Let $\mathcal{G}$ be a Rees algebra on $V$. A \textit{$\mathcal{G}$-permissible blow up}
$$V\stackrel{\pi }{\leftarrow} V_1\mbox{,}$$
is the blow up of $V$ at a smooth closed subset $Y\subset V$ contained in $\mathrm{Sing}(\mathcal{G})$ (a {\em permissible center for $\mathcal{G}$}). We denote then by $\mathcal{G}_1$ the (weighted) transform of $\mathcal{G}$ by $\pi $, which is defined as
$$\mathcal{G}_1:=\bigoplus _{l\in \mathbb{N}}I_{l,1}W^l\mbox{,}$$
where 
\begin{equation}\label{eq:transf_law}
I_{l,1}=I_l\mathcal{O}_{V_1}\cdot I(E)^{-l}
\end{equation}
for $l\in \mathbb{N}$ and $E$ the exceptional divisor of the blow up $V\longleftarrow V_1$.
\end{Def}

  As we will see in section \ref{Presentaciones_locales}, the problem of simplification of the maximum multiplicity of an algebraic variety can  be translated into the problem of {\em resolution of a suitably defined Rees algebra}. This motivates the following definition (see also Example \ref{ejemplo_hipersuperficie} below).

\begin{Def}\label{def:res_RA}
Let $\mathcal{G}$ be a Rees algebra over $V$. A \textit{resolution of $\mathcal{G}$} is a finite sequence of transformations
\begin{equation}\label{diag:res_Rees_algebra}
\xymatrix@R=0pt@C=30pt{
V=V_0 & V_1 \ar[l]_>>>>>{\pi _1} & \ldots \ar[l]_{\pi _2} & V_l \ar[l]_{\pi _l}\\
\mathcal{G}=\mathcal{G}_0 & \mathcal{G}_1 \ar[l] & \ldots \ar[l] & \mathcal{G}_l \ar[l]
}\end{equation}
at permissible centers $Y_i\subset \text{Sing} ({\mathcal G}_i)$, $i=0,\ldots, l-1$, such that $\mathrm{Sing}(\mathcal{G}_l)=\emptyset$, and such that the exceptional divisor of the composition $V_0\longleftarrow V_l$ is a union of hypersurfaces with normal crossings. Recall that a set of hypersurfaces $\{H_1,\ldots, H_r\}$ in a smooth $n$-dimensional $V$  has normal crossings at a point $\xi\in V$ if there is a regular system of parameters $x_1,\ldots, x_n\in {\mathcal O}_{V, \xi}$ such that if $\xi\in H_{i_1}\cap \ldots \cap H_{i_s}$,  and $\xi \notin H_l$ for $l\in \{1,\ldots, r\}\setminus \{i_1,\ldots,i_s\}$, then
 ${\mathcal I}(H_{i_j})_{\xi}=\langle x_{i_j}\rangle$  for $i_j\in \{i_1,\ldots, i_s\}$;
we say that $H_1,\ldots, H_r$ have normal crossings in V if they have normal crossings at each point of $V$. 
\end{Def}

\begin{Ex}\label{ejemplo_hipersuperficie}
With the setting of Example~\ref{Ex:HiperRees}, a resolution of the Rees algebra $\mathcal{G}=R[fW^b]$ gives a sequence of transformations such the multiplicity of the strict transform of $X$  has decreased:
\begin{gather*}
\xymatrix@R=0pt@C=30pt{
\mathcal{G}=\mathcal{G}_0 & \mathcal{G}_1 \ar[l] & \ldots \ar[l] & \mathcal{G}_{l-1} \ar[l] & \mathcal{G}_l \ar[l]\\
V=V_0 & V_1 \ar[l]_>>>>>{\pi _1} & \ldots \ar[l]_{\pi _2} & V_{l-1}\ar[l]_{\pi_{l-1}} & V_l \ar[l]_{\pi _l}\\
\  \; \; \; \; \; \cup & \cup & & \cup & \cup \\
X=X_0 & X_1 \ar[l]_>>>>>{\pi _1} & \ldots \ar[l]_{\pi _2} & X_{l-1} \ar[l]_{\pi_{l-1}} & X_l \ar[l]_{\pi _l}
} \\
b=\max\mult(X_0) =  \max\mult(X_1) =\cdots = \max\mult(X_{l-1})>\max\mult(X_l).
\end{gather*}
Here  each $X_i$ is the strict transform of $X_{i-1}$.
Note that the set of points of $X_l$ having multiplicity $b$ is $\Sing(\mathcal{G}_{l})=\emptyset$ (see \ref{LocalPresMult}).
\end{Ex}

\begin{Rem}
Resolution of Rees algebras is known to exists when $V$ is defined over a field of characteristic zero  (\cite{Hir}, \cite{Hir1}). In \cite{V1} and \cite{B-M} different algorithms of resolution of Rees algebras are presented (see also \cite{E_V97}, \cite{E_Hau}). An algorithmic  resolution  requires the definition of invariants associated to the points of the singular locus of a given Rees algebra so as to define a stratification of this closed set. The most important invariant involved in the resolution process is {\em Hironaka's order function}. 
\end{Rem}

\begin{Parrafo}\label{Def:HirOrd}
\textbf{Hironaka's order of a Rees Algebra.}
(\cite[Proposition 6.4.1]{E_V}) Let $\G$ be an ${\mathcal O}_V$-Rees algebra. 
We define the \textit{order of an element $fW^{n}\in \mathcal{G}$ at $\xi \in \mathrm{Sing}(\mathcal{G})$} as
$$\mathrm{ord}_{\xi }(fW^{n}):=\frac{\nu _{\xi }(f)}{n}\mbox{.}$$
We define the \textit{order of the Rees algebra $\mathcal{G}$ at $\xi \in \mathrm{Sing}(\mathcal{G})$} as the infimum of the orders of the elements of $\mathcal{G}$ at $\xi$, that is
$$\mathrm{ord}_{\xi }(\mathcal{G}):=\inf _{l\geq 0}\left\{ \frac{\nu_{\xi }(I_l)}{l}\right\} \mbox{.}$$
 This is what we call {\em Hironaka's order function of $\G$ at the point $\xi$}. 
If $\mathcal{G}=R[f_1W^{n_1},\ldots ,f_rW^{n_r}]$ and $\xi\in \mathrm{Sing}(\mathcal{G})$ then it can be shown (see \cite[Proposition 6.4.1]{E_V}) that: 
$$\mathrm{ord}_{\xi }(\mathcal{G})=\min _{i=1\ldots r}\left\{ \mathrm{ord}_{\xi }(f_iW^{n_i})\right\} \mbox{.}$$
\end{Parrafo}

The following two definitions  correspond to operations that  can be performed on  a given Rees algebra without changing  the singular locus   and Hironaka's order function. In fact, as we will see, Rees algebras linked by the these operations share the same algorithmic resolution (at least in characteristic zero). 

\begin{Def} \label{Diff_saturation} 
A Rees algebra $\mathcal{G}=\oplus _{l\geq 0}I_lW^l$ over $V$ is \textit{differentially closed} (or a \textit{Diff-algebra}) if there is an affine open covering $\{U_i\}_{i\in I}$ of $V$,  such that for every $D\in \mathrm{Diff}^{r}(U_i)$ and $h\in I_l(U_i)$, we have $D(h)\in I_{l-r}(U_i)$ whenever $l\geq r$ (where $\mathrm{Diff}^{r}(U_i)$ is the locally free sheaf of $k$-linear differential operators of order less than or equal to $r$). In particular, $I_{l+1}\subset I_l$ for $l\geq 0$. We denote by $\mathrm{Diff}(\mathcal{G})$ the smallest differential Rees algebra containing $\mathcal{G}$ (its \textit{differential closure}). (See \cite[Theorem 3.4]{V07} for the existence and construction.)   If $\beta: V\to V'$ is a smooth morphism, then we will say that $\G$ has a  {\em  $\beta$-relative differential} structure if $\G$ is closed by the action of the relative differential operators in  $\Diff_{V/V'}$.  
\end{Def}

\begin{Rem}(\cite[proof of Theorem 3.4]{V07}) If $\mathcal{G}$ is a Rees algebra over a smooth $V$, locally generated by a set $\left\{ f_1W^{n_1},\ldots ,f_rW^{n_r} \right\} \subset \mathcal{G}$, then $\mathrm{Diff}(\mathcal{G})$ is (locally) generated by the set 
$$\left\{ D(f_i)W^{n_i-\alpha } : D\in \mathrm{Diff}^{\alpha }\mbox{,\, }0\leq \alpha <n_i\mbox{,\, }i=1,\ldots ,r\right\} \mbox{.}$$
\end{Rem}

\vspace{0.2cm}

\begin{Def}
Two Rees algebras over a ring $R$ (not necessary smooth) are \textit{integrally equivalent} if their integral closure in $\mathrm{Quot}(R)[W]$ coincide. We say that a Rees algebra over $R$, $\mathcal{G}=\oplus _{l\geq 0}I_lW^l$ is \textit{integrally closed} if it is integrally closed as an $R$-ring in $\mathrm{Quot}(R)[W]$. We denote by $\overline{\mathcal{G}}$ the integral closure of $\mathcal{G}$.
\end{Def}

\begin{Rem} \label{order_int_diff} If $R$ is smooth over a perfect field $k$, then 
for a Rees algebra $\mathcal{G}\subset R[W]$ we have that
$\Sing(\mathcal{G})=\Sing(\overline{\mathcal{G}})=\Sing(\Diff(\mathcal{G}))$   (see \cite[Proposition 4.4 (1), (3)]{V3}). 
In fact for any point $\xi\in\Sing(\mathcal{G})$ we have
$\ord_{\xi}(\mathcal{G})=\ord_{\xi}(\overline{\mathcal{G}})=\ord_{\xi}(\Diff(\mathcal{G}))$  
(see   \cite[Remark 3.5, Proposition 6.4 (2)]{E_V}).       
\end{Rem}

\section{Local presentations} \label{Presentaciones_locales}

Let $X$ be an equidimensional algebraic variety of dimension $d$   defined over a perfect field $k$.
Consider the multiplicity function
\begin{align*}
\mult_X: X & \longrightarrow \mathbb{N} \\
\xi & \longrightarrow \mult_X(\xi)=\text{mult}_{{\mathfrak m}_{\xi}} {\mathcal O}_{X,\xi}
\end{align*}
where $\text{mult}_{{\mathfrak m}_{\xi}} {\mathcal O}_{X,\xi}$ stands for the multiplicity of the local ring $\mathcal{O}_{X,\xi}$ at the maximal ideal $\mathfrak{m}_{\xi}$.

It is known that the function $\mult_X$ is upper-semi-continuous (see \cite{Dade}). In particular, if $m=\max\mult_X$ is the maximum value of the multiplicity of $X$ then the set
$$\mathrm{\underline{Max}}\mult_X=\left\lbrace \xi \in X \mid \mult_X(\xi)\geq m\right\} =
\left\{ \xi \in X\mid \mult_{X}(\xi)=m\right\rbrace $$
is closed. It is also known  that the multiplicity function can not increase after a blow up $\phi:X'\to X$ with   regular   center $Y$
provided that $Y\subset \mathrm{\underline{Max}}\mult_X$ (cf. \cite{Dade}).  This means that $\mult_{X'}(\xi')\leq \mult_X(\xi)$ for  
$\xi=\phi(\xi')$, $\xi '\in X'$.

One could try to approach  a resolution of singularities by defining a sequence of blow ups   at regular equimultiple centers  
\begin{equation}
\label{simple_1}
\xymatrix{X=X_0 & X_1\ar[l] & \ar[l] \ldots & \ar[l] X_{l-1} & \ar[l] X_l}
\end{equation}
so that 
\begin{equation}
\label{simple_2}
m=\max\mult_{X_0}=\max\mult_{X_1} =\ldots =\max\mult_{X_{l-1}} >\max\mult_{X_l}. 
\end{equation}
  A sequence like (\ref{simple_1})  with the property (\ref{simple_2})  is a {\em simplification of the multiplicity of $X$}.  
\medskip

A {\em local presentation for the multiplicity}  is an expression of  the closed set $\left\{ \xi \in X\mid \mult_{X}(\xi)=m\right\rbrace$ in terms of the maximum multiplicity locus of a suitably chosen finite set of hypersurfaces defined in a smooth ambient space. This information is much   easier to handle (see Theorem \ref{Th:PresFinita} and \ref{PresFinita}).
These hypersurfaces will be defined in a suitable embedding of $X$ in a smooth space $V$.
Moreover we will require that this presentation  holds after  certain transformations that we specify  in the next definition:

\begin{Def} \label{Def:LocalSeq}
Let $V$ be a smooth scheme   defined over a perfect field $k$. 
A \emph{permissible transformation} is either: 
\begin{itemize}
\item A   permissible   blow up $V_1\to V$, i.e., the blow up at a smooth center $Y\subset V$; or

\item A smooth morphism $V_1\to V$.
\end{itemize}
A \emph{local sequence} is a sequence of permissible transformations, 
$$V=V_0\stackrel{\phi _1}{\longleftarrow} V_1\stackrel{\phi _2}{\longleftarrow}\ldots \stackrel{\phi _l}{\longleftarrow} V_l,$$
so that each $\phi _j$, $j=1,\ldots,l$, is either a permissible blow up at $Y_{i-1}\subset V_{i-1}$ or a smooth morphism.
\end{Def}

\begin{Def} \label{Def:GLocSeq}
Let $\mathcal{G}$ be a Rees algebra on a smooth scheme  $V$    over a perfect field $k$.  
A $\mathcal{G}$-local sequence is a  local    sequence as in Definition \ref{Def:LocalSeq},  
\begin{equation*}
\xymatrix@R=0pt@C=30pt{
V=V_0 & V_1 \ar[l]_>>>>>{\phi _1} & \ldots \ar[l]_{\phi _2} & V_l \ar[l]_{\phi _l}\\
\mathcal{G}=\mathcal{G}_0 & \mathcal{G}_1 \ar[l] & \ldots \ar[l] & \mathcal{G}_l, \ar[l]
}\end{equation*}
such that for every $i=1,\ldots,l$, 
\begin{itemize}
\item If $\phi_i$ is a blow up then $Y_{i-1}\subset \Sing  (\mathcal{G}_{i-1})$ and $\mathcal{G}_i$ is the transform of $\mathcal{G}_{i-1}$ as in Definition \ref{def:transf_law};

\item If $\phi_i$ is a smooth morphism then $\mathcal{G}_{i}$ is the pull-back of $\mathcal{G}_{i-1}$.
\end{itemize}
\end{Def}

\begin{Parrafo} \label{Def:LocPres} {\bf Local presentations for the mutiplicity.}   Let $X$ be an algebraic variety defined over a perfect field $k$,    and  let $m=\max\mult_X > 1$.
A \emph{global presentation} for the function $\mult_X$ is given by:  

\begin{enumerate}
\item[(i)] A closed embedding $X\subset V^{}$ where $V^{}$ is a smooth scheme of dimension $n >d$; 
\item[(ii)]  A collection of hypersurfaces $\mathfrak{H}_1,\mathfrak{H}_2,\ldots,\mathfrak{H}_r$ in $V^{}$, and weights
$b_1,b_2,\ldots,b_r\in\mathbb{N}$ with  $\text{max}\mult_{\mathfrak{H}_i}=b_i$ for $i=1,\ldots,r$ such that:    
\begin{enumerate}
\item[(a)] The closed set $\mathrm{\underline{Max}}\mult_X$ can  be expressed in terms of hypersurface multiplicities:
\begin{equation} \label{eq:LocPresHyp}
\mathrm{\underline{Max}}\mult_X=
\left\lbrace \xi\in V^{} \mid \mult_{\mathfrak{H}_i}(\xi)\geq b_i, \ i=1,2,\ldots,r \right\rbrace=\cap_{i=1}^r\text{\underline{Max} mult}_{\mathfrak{H}_i}; 
\end{equation}
\item[(b)]   Expression (\ref{eq:LocPresHyp}) is stable under local sequences: given any local sequence as in 
Definition~\ref{Def:LocalSeq}: 
\begin{equation} \label{eq:LocalSeq}
\xymatrix@R=0pt@C=30pt{
V^{}=V^{}_0 & V^{}_1 \ar[l]_>>>>>{\phi _1} & \ldots \ar[l]_{\phi _2} & V^{}_l \ar[l]_{\phi _l}\\
\  \; \; \; \; \; \cup & \cup & & \cup \\
X=X_0 & X_1 \ar[l] & \ldots \ar[l] & X_l \ar[l]
}\end{equation}
(where for $j=1,\ldots,l$, $X_j$ is the strict transform of $X_{j-1}$ and if $\phi_j$ is a blow up then the center is contained in 
$\mathrm{\underline{Max}}\mult_{X_{j-1}}$), then for $ j=0,1,\ldots,l$, 
\begin{equation} \label{eq:LocPresHyp2}
\left\lbrace \xi\in X_j \mid \mult_{X_j}(\xi)=m \right\rbrace=
\left\lbrace \xi\in V^{}_j \mid \mult_{\mathfrak{H}_{i,j}}(\xi)\geq b_i, \ i=1,2,\ldots,r \right\rbrace,
\end{equation}
where $\mathfrak{H}_{i,j}$ is the strict transform of $\mathfrak{H}_{i,j-1}$ in $V_j$ ($\mathfrak{H}_{i,0}=\mathfrak{H}_{i}$).
\end{enumerate} 
\end{enumerate}
A \emph{local presentation} for the function $\mult_X$   in a neighbourhood of  a point  $\xi\in \mathrm{\underline{Max}}\mult_X$ is a presentation satisfying conditions (i) and (ii) in  a suitable  (\'etale) open neighborhood  $U\subset X$   of $\xi$.
\end{Parrafo}

\begin{Rem}
Note that equality (\ref{eq:LocPresHyp2}) is equivalent to saying that for $j=0,1,\ldots,l-1$, 
$$
\mathrm{\underline{Max}}\mult_{X_{j}}=
\left\lbrace \xi\in V^{}_j \mid \mult_{\mathfrak{H}_{i,j}}(\xi)\geq b_i, \ i=1,2,\ldots,r \right\rbrace,
\qquad 
$$
and either $\left\lbrace \xi\in X_l \mid \mult_{X_l}(\xi)=m \right\rbrace=\emptyset$
(which means $\max\mult_{X_l}<m$), or 
$$
\mathrm{\underline{Max}}\mult_{X_{l}}=
\left\lbrace \xi\in V^{}_l \mid \mult_{\mathfrak{H}_{i,l}}(\xi)\geq b_i, \ i=1,2,\ldots,r \right\rbrace.
$$
\end{Rem}

\begin{Parrafo} \label{LocalPresMult}
\textbf{Rees algebras vs. local presentations.} Let  $X$ be an algebraic variety, let  $\xi\in \mathrm{\underline{Max}}\mult_X$ and suppose that there is a local presentation as in \ref{Def:LocPres} in an (\'etale) neighborhood  $U\subset X$ of $\xi$  which we denote again by $X$ for simplicity.     Then we may assume that 
$V=\Spec(R)$ for some smooth $k$ algebra $R$, and that  each hypersurface $\mathfrak{H}_i$ is defined by an equation $f_i\in R$, $i=1,\ldots,r$. Now, if we define  the $R$-Rees algebra, $\mathcal{G}=R[f_1W^{b_1},\ldots,f_r W^{b_r}]$, 
then the equality (\ref{eq:LocPresHyp})  can be  expressed as:   
\begin{equation} \label{eq:LocPresSing}
\mathrm{\underline{Max}}\mult_X=\Sing(\mathcal{G}). 
\end{equation}
Moreover, given a local sequence as in Definition \ref{eq:LocalSeq}, there is an induced $\G$-local sequence  and transformations of Rees algebras as in Definition \ref{Def:GLocSeq}, 
\begin{equation} \label{eq:LocalGSeq}
\xymatrix@R=0pt@C=30pt{
V=V_0 & V_1 \ar[l]_>>>>>{\phi _1} & \ldots \ar[l]_{\phi _2} & V_l \ar[l]_{\phi _l}\\
\  \; \; \; \; \; \cup & \cup & & \cup \\
X=X_0 & X_1 \ar[l] & \ldots \ar[l] & X_l \ar[l] \\
\mathcal{G}=\mathcal{G}_0 & \mathcal{G}_1 & \ldots  & \mathcal{G}_l
}\end{equation}
 and equality (\ref{eq:LocPresHyp2}) can be expressed as
\begin{equation} \label{eq:LocPresSing2}
\left\lbrace \xi\in X_j \mid \mult_{X_j}(\xi)=m \right\rbrace=
%\mathrm{\underline{Max}}\mult_{X_{j}}=
\Sing(\mathcal{G}_j), 
\qquad j=0,1,\ldots,l.
\end{equation}

From the previous discussion it follows  that finding a local presentation for the function $\mult_X$ at a point $\xi$ is equivalent to choosing a local (\'etale) embedding $X\subset V$ and a Rees algebra $\mathcal{G}$ in $V$ such that:
\begin{itemize}
\item $\mathrm{\underline{Max}}\mult_X=\Sing(\mathcal{G})$; 

\item For any local sequence as in (\ref{eq:LocalSeq}) or  (\ref{eq:LocalGSeq}) we have 
$$\left\lbrace \xi\in X_j \mid \mult_{X_j}(\xi)=m \right\rbrace=
%\mathrm{\underline{Max}}\mult_{X_{j}}=
\Sing(\mathcal{G}_j), \qquad j=0,1,\ldots,l.$$
\end{itemize}
As a consequence of the previous discussion, the problem of finding a simplification of the multiplicity of an algebraic variety can be translated into the problem of finding a resolution of a suitable Rees algebra in a smooth scheme.   In what follows, we will use the notation $(V,\G)$ for a given local presentation of the multiplicity as above. 
\end{Parrafo}

\begin{Thm} \label{Th:PresFinita}
\cite[7.1]{V}
Let $X$ be a reduced equidimensional scheme   defined over a perfect field $k$.
For every point $\xi\in X$ there exists a local presentation for the function $\mult_X$. 
\end{Thm}
In the following lines we present some of the ideas on which   the proof of Theorem \ref{Th:PresFinita} is based. We will be using some of them in the proof of Theorem \ref{Main_Theorem}.  
 
\begin{Parrafo} \label{PresFinita}
\textbf{Some ideas behind the proof of Theorem \ref{Th:PresFinita}.\cite[\S5, \S7]{V}}
Let $X=\Spec(B)$ be an affine   algebraic variety of dimension $d$ defined over a perfect field $k$,  and let $\xi\in \underline{\text{Max}} \text{ Mult}_X$.
Then  it can be shown that (maybe, after replacing $B$ and  $k$  by  suitable  \'etale extensions), 
there is a regular   $k$-algebra   $S$ and a finite extension $S\subset B$ of generic rank $m=\text{max mult}_X$, inducing a finite morphism $\alpha: \text{Spec}(B)\to \text{Spec}(S)$. Under these assumptions, $B=S[\theta_1,\ldots, \theta_{n-d}]$,  for some $\theta_1,\ldots, \theta_{n-d}\in B$ and some $n>d$. Observe that the previous extension induces a natural embedding $X\subset V^{(n)}:=\text{Spec}(R)$, where $R=S[x_1,\ldots, x_{n-d}]$.  Let $K(S)$, repectively $K(B)$, be the total rings of fractions  of $S$ and $B$.  Now, if  $f_i(x_i)\in K(S)[x_i]$ denotes the minimal polynomial of $\theta_i$ for $i=1,\ldots, (n-d)$, then it can be checked that in fact  $f_i\in S[x_i]$ and as a consequence $\langle f_1(x_1), \ldots, f_{n-d}(x_{n-d})\rangle\subset {\mathcal I}(X)$, the defining ideal of $X$ in $V^{(n)}$.  If each polinomial $f_i$ is of degree $m_i$,  it is proven that the differential Rees algebra
\begin{equation}
\label{G_Representa}
\Gn:=\Diff(R[f_1W^{m_1}, \ldots, f_{n-d}W^{m_{n-d}}])
\end{equation}
is a local presentation of $\underline{\text{Max}} \text{ mult}_X$   at $\xi$. Moreover, for each $i=1,\ldots ,n-d$, there is a commutative diagram:  
\begin{equation}\label{diagrama_presentacion}
\xymatrix{
 S[x_1,\ldots ,x_{n-d}] \ar[r] &     S[x_1,\ldots ,x_{n-d}]  /\langle f_1,\ldots, f_{n-d}\rangle \ar[r] & B \\
 S[x_i]  \ar[u]   \ar[r] &     B_i=S[x_i]/\langle f_i \rangle  \ar[u] \ar[ur] &    \\
 S \ar@/^2pc/[uu]^{\beta^*}  \ar[u]_{\beta_{H_i}^*}   \ar[ur]^{\alpha_{H_i}^*}   
\ar@/_2pc/[uurr]_{\alpha^*}  
& & 
}
\end{equation}
 
The inclusion $S\subset S[x_i]/\langle f_i\rangle$ induces a finite projection $\alpha_{H_i}: \text{Spec} (B_i)\to \text{Spec}(S)$ and $\G _i^{(d+1)}=\Diff (S[x_i][f_iW^{b_i}])\subset S[x_i][W]$ represents the multiplicity of  the hypersurface $H_i$ defined by $f_i$ in $V_i^{(d+1)}=\Spec(S[x_i])$. 

Finally, since the generic rank of the extension $S\subset B$ equals $m=\text{max mult}_X$, by Zariski's multiplicity formula for finite projections (cf., \cite[Chapter 8, \S 10, Theorem 24]{Z-SII}) it follows that: 
\begin{enumerate}
	\item The point $\xi$ is the unique point in the fiber over $\alpha(\xi)\in \text{Spec}(S)$; 
	\item The residue fields at $\xi$ and $\alpha(\xi)$ are isomorphic; 
	\item The defining  ideal of $\alpha(\xi)$ at $S$, $\mathfrak{m}_{\alpha(\xi)}$, generates a reduction of the maximal ideal of $\xi$, $\mathfrak{m}_{\xi}$, at $B_{\mathfrak{m}_{\xi}}$. 
\end{enumerate}
\end{Parrafo}

\begin{Rem}
In fact, the notion of local presentation as in (\ref{Def:LocPres}) can be given  for any upper-semi-continuous function on $X$, as long as the value of the function does not increase after the blow up at a smooth  center included in the stratum defined by the maximum value of the function.
\medskip

An example of a function having this property is the Hilbert-Samuel function,
$$\mathrm{HS}_X:X\to \mathbb{N}^{\mathbb{N}}$$
which is upper-semi-continuous  (see \cite{Bennett}); if   $\phi:X'\to X$  is the blow up   at  smooth center $Y\subset X$ such that the Hilbert-Samuel function is constant along $Y$ then  we have that   (see \cite{Hir2}),
$$\mathrm{HS}_{X'}(\xi')\leq \mathrm{HS}_X(\phi(\xi)), \qquad \forall \xi'\in X'.$$

Indeed, local presentations for the Hilbert-Samuel function also exist and, in characteristic zero, they are used by Hironaka to obtain resolution of singularities (see \cite{Hir1}).
\end{Rem}

Local presentations are not unique. For instance, once a local embedding $X\subset V$  is fixed, there may be different  ${\mathcal O}_V$-Rees algebras representing $\Mm$. However, it can be proven that they all lead to the same simplification of the multiplicity of $X$ (if it exists). This fact will be clarified in forthcoming paragraphs (see Corollary \ref{cor:res_weak_eq}). The previous discussion  motivates the next definition.

\begin{Def}\cite[Definition 3.5]{Br_G-E_V} Let $V$ be a smooth scheme over a perfect field $k$. 
We say that two $\mathcal{O}_{V}$-Rees algebras $\mathcal{G}$ and $\mathcal{H}$ are \textit{weakly equivalent} if:

\begin{enumerate}
	\item $\mathrm{Sing}(\mathcal{G})=\mathrm{Sing}(\mathcal{H})$;
	\item Any $\mathcal{G}$-local sequence over $V$
	$$\mathcal{G}=\mathcal{G}_0\longleftarrow \mathcal{G}_1\longleftarrow \ldots \longleftarrow \mathcal{G}_r$$
	induces an $\mathcal{H}$-local sequence over $V$
	$$\mathcal{H}=\mathcal{H}_0\longleftarrow \mathcal{H}_1\longleftarrow \ldots \longleftarrow \mathcal{H}_r$$
	and vice versa, and moreover the equality in (1) is preserved, that is
	\item $\mathrm{Sing}(\mathcal{G}_j)=\mathrm{Sing}(\mathcal{H}_j)$ for $j=0,\ldots ,r$.\\
\end{enumerate}
\end{Def}

\begin{Rem}
\ 
\begin{itemize}
\item \cite[Proposition 5.4]{E_V}
If $\mathcal{G}_1$ and $\mathcal{G}_2$ are two integrally equivalent Rees algebras over $R$, then they are weakly equivalent.

\item \cite[Section 4]{Br_G-E_V}
A Rees algebra $\mathcal{G}$ and its differential closure $\mathrm{Diff}(\mathcal{G})$ are weakly equivalent. This is a consequence of Giraud's Lemma (see \cite{Giraud}).

\item \cite[Theorem 3.11]{Br_G-E_V}
Let $\mathcal{G}_1$ and $\mathcal{G}_2$ be two Rees algebras over $V$. Then $\mathcal{G}_1$ and $\mathcal{G}_2$ are weakly equivalent if and only if $\overline{\mathrm{Diff}(\mathcal{G}_1)}=\overline{\mathrm{Diff}(\mathcal{G}_2)}$.
\end{itemize}
In fact, from Remark \ref{order_int_diff} it follows now that: 
\end{Rem}

\begin{Cor}\label{cor:ord_weak_eq}
Let $\mathcal{G}_1$ and $\mathcal{G}_2$ be two weakly equivalent Rees algebras over $V$. Then for all $\eta \in \mathrm{Sing}(\mathcal{G}_1)=\mathrm{Sing}(\mathcal{G}_2)$, we have $\mathrm{ord}_{\eta }\mathcal{G}_1=\mathrm{ord}_{\eta }\mathcal{G}_2$.
\end{Cor}

As a consequence: 

\begin{Cor} \cite[Remark 11.8]{Br_V2} \label{cor:res_weak_eq}
Let $\mathcal{G}_1$ and $\mathcal{G}_2$ be two weakly equivalent Rees algebras. Then a constructive resolution of $\mathcal{G}_1$ induces a constructive resolution of $\mathcal{G}_2$ and vice versa.
\end{Cor}

Corollary \ref{cor:res_weak_eq} follows from Corollary \ref{cor:ord_weak_eq} and the fact that, in characteristic zero,  constructive resolution of Rees algebras is  given in terms of the so called {\em satellite functions}. All such  functions  derive form Hironaka's order function (see \cite{E_V97}).

\section{Elimination algebras}\label{ElimAlg}

  As indicated in the previous section, the problem of algorithmic simplification of the   multiplicity of an algebraic variety (and hence, that of algorithmic resolution) can be, ultimately, translated into a problem of resolution of Rees algebras via local presentations (see  \ref{LocalPresMult}). 
 Now suppose we are given a Rees algebra ${\mathcal G}$ on a smooth $n$-dimensional scheme $V$. Sometimes the resolution of $\G$ is {\em equivalent} to the resolution of another Rees algebra defined on a smooth scheme of lower dimension,   the latter, at least phylosophically, should be an easier problem to solve. 

\

For instance, let $k$ be a perfect field, and consider the Rees algebra ${\mathcal G}$ generated by $xW, y^3W^2$ over $V=\text{Spec} (k[x,y])$. Notice that there is a natural inclusion $k[y]\subset k[x,y]$ inducing a smooth projection  $\beta: \text{Spec} (k[x,y]) \to \text{Spec} (k[y])$. Set $Z=\text{Spec} (k[y])$. Now consider the Rees algebra   ${\mathcal R}=\G\cap k[y]=k[y][y^3W^2]$.  It can be checked that  $\Sing (\G)$ is homeomorphic to $\Sing (\R)$ via $\beta$. Moreover, both algebras are linked in a stronger way. It can be shown that any $\G$-local sequence over $V$ (as in Definition \ref{Def:GLocSeq}) induces an $\R$-local sequence over $Z$, together with vertical smooth projections,
\begin{equation} \label{Eq:SeqEliminaZR}
\xymatrix@R15pt{ (V_0, \G_0)= (V, \G) \ar[d]^{\beta} & \ar[l]  (V_1, \G_1)   \ar[d]^{\beta_1}  &  \ar[l] \ldots &  \ar[l]  (V_{m-1}, \G_{m-1} )   \ar[d]^{\beta_{m-1}}  & \ar[l]  (V_{m}, \G_{m} )   \ar[d]^{\beta_m} \\
(Z_0, \R_0)= (Z, \R) & \ar[l]  (Z_1, \R_1)  &  \ar[l] \ldots &  \ar[l]  (Z_{m-1}, \R_{m-1} )  & \ar[l]  (Z_{m}, \R_{m} )}
\end{equation}
and transformations of Rees algebras so that 
$\Sing (\G_i)$ is homeomorphic to $\Sing (\R_i)$ via $\beta_i$ for $i=1,\ldots, m$ (it is worth noticing that for the diagram to commute we may have to replace the transform of $V_i$, $V_{i+1}$,  by a suitable open subset containing $\Sing(\G_{i+1})$ for those $V_i\leftarrow V_{i+1}$ that correspond to blow ups).
Similarly, it can be shown that  any $\R$-local sequence on $Z$ induces a $\G$-local sequence on $V$ together with commutative diagrams as in (\ref{Eq:SeqEliminaZR}) and with the same properties as before.
Thus it follows that finding a resolution of $\G$ is equivalent to finding a resolution of $\R$, but this last problem is easier to solve. 

\

We would like to generalize   the previous setting to a more general one. Here is were elimination algebras come into play. In the following paragraphs we will explain how one can proceed to define an elimination algebra from a given one in a lower dimensional scheme (whenever certain technical conditions are satisfied).   As we will see,  in the previous example, $\R$ above is an elimination algebra of $\G$ over $Z$.

\

Suppose $V^{(n)}$ is an $n$-dimensional smooth scheme over a perfect field $k$, and let $\G^{(n)}$ be a Rees algebra over $V^{(n)}$.  As a first step to define an elimination algebra, given a suitable integer $e\geq 1$, we will search for smooth morphisms  from $V^{(n)}$ to some $(n-e)$-dimensional smooth scheme so that $\Sing (\Gn)$ be homeomorphic to its image via $\beta$. One way to accomplish this condition is by considering morphisms from $V^{(n)}$ which are somehow {\em transversal}  to  $\Gn$. The condition of transversality is expressed in terms of the {\em tangent cone of $\Gn$} at a given point of its singular locus (see Definition \ref{def83} below).   

\

Let    $\xi \in \Sing (\G^{(n)})$ be a closed point, and let $\text{Gr}_{{\mathfrak m}_{\xi}}({\mathcal O}_{V^{(n)},\xi})$ denote the graded ring of $\mathcal{O}_{V^{(n)},\xi}\simeq k'[Y_1,\ldots, Y_n]$, where $k'$ denotes the residue field at $\xi$.  Recall that $\text{Spec} (\text{Gr}_{{\mathfrak m}_{\xi}}({\mathcal O}_{V^{(n)},\xi}))={\mathbb T}_{{V^{(n)}},\xi}$, the tangent space of $V^{(n)}$ at $\xi$. 

\begin{Def}\label{tangent_cone}  {\rm Suppose $\xi\in \Sing(\Gn)$ is a closed point with $\ord_{\xi}(\Gn)=1$. The {\em initial ideal} or  {\em tangent ideal of $\G^{(n)}$ at $\xi$}, is defined as the homogeneous ideal of $\text{Gr}_{{\mathfrak m}_{\xi}}({\mathcal O}_{V^{(n)},\xi})$ generated by
\[
	\text{In}_{\xi}(I_l) := \frac{ I_l+{\mathfrak m}^{l+1}_{\xi}}{{\mathfrak m}^{l+1}_{\xi}}
\]
for all $l \geq 1$, and it is denoted by $\text{In}_{\xi}\G^{(n)}$. The {\em tangent cone of    $\Gn$ at $\xi$}  is the closed subset of ${\mathbb T}_{{V^{(n)}},\xi}$ defined by the initial ideal   of    $\Gn$ at $\xi$, and it is denoted by ${\mathcal C}_{\Gn,\xi}$. }
\end{Def}

\begin{Def}\label{tau}  \cite[4.2]{V07} 
Let $\Gn$ and $\xi$ be as in Definition \ref{tangent_cone}.  The {\em $\tau$-invariant of ${\Gn}$ at ${\xi}$} is the minimum number of variables  in $\text{Gr}_{{\mathfrak m}_{\xi}}({\mathcal O}_{V^{(n)},\xi})$   needed to generate 
$\operatorname{In}_{{\xi}}({\Gn})$. This in turn is the codimension of the largest linear
subspace  ${\mathcal L}_{{\Gn},{\xi}}\subset {\mathcal
C}_{{\Gn},{\xi}}$ such that $u+v\in {\mathcal C}_{{\Gn},{\xi}}$ for all $u\in {\mathcal C}_{{\Gn},{\xi}}$ and $v\in
{\mathcal L}_{{\Gn},{\xi}}$. The $\tau$-invariant of ${\Gn}$ at ${\xi}$ is denoted by  $\tau_{{\Gn},{\xi}}$. 
\end{Def} 

\begin{Rem} \label{tau_properties} Note that: 
\begin{enumerate} 
\item The ideal $\text{In}_{\xi} (\G^{(n)})$ can be defined at any point $\xi\in \Sing(\Gn)$, however it  is non zero if and only if $\ord_{\xi}(\Gn)=1$. It is in this case when the $\tau$-invariant is defined.  Moreover, for  $\xi\in \Sing(\Gn)$ it can be checked that $\ord_{\xi}(\Gn)=1$ if and only if $\tau_{{\Gn},{\xi}}\geq 1$.

\item Since  ${\Gn}\subset \Diff(\Gn)$, there is    an inclusion 
${\mathcal
C}_{\Diff({\Gn}),{\xi}} \subset {\mathcal
C}_{{\Gn},{\xi}}$. Moreover, 
 ${\mathcal C}_{\Diff({\Gn}),{\xi}}={\mathcal L}_{\Diff({\Gn}),{\xi}}={\mathcal L}_{{\Gn},{\xi}}$.  
 In particular, $\Gn$, $\overline{\Gn}$, and $\Diff({\Gn})$ share  the same $\tau$-invariant at any point $\xi\in \Sing (\Gn)$ (see for instance \cite[Remark 4.5, Theorem 5.2]{B}). 
\end{enumerate} 

\end{Rem}

\begin{Def}\label{def83}
{\rm Let $\Gn$ be a Rees algebra on a smooth
$n$-dimensional scheme $V^{(n)}$ over a perfect  field $k$, and let $\xi\in
\mbox{Sing }{\Gn}$ be a  closed point with $\tau_{{\Gn},{\xi}}\geq e\geq 1$. We   say that a local
smooth projection to a  $(n-e)$-dimensional (smooth) scheme  $V^{(n-e)}$, say $\beta : V^{(n)}   \to  V^{(n-e)}$, 
is {\em  $\Gn$-admissible locally at $\xi$} if the
following conditions hold:
\begin{enumerate}
\item The  point $\xi$ is not contained in any codimension-$e$-component  of    $\mbox{Sing }{\Gn}$;
  \item The Rees algebra ${\Gn}$ is a $\beta$-relative differential
algebra (see Definition \ref{Diff_saturation});  
\item {\em Transversality:} $\ker (d_{\xi}\beta)\cap   {\mathcal C}_{{\Gn},\xi}=\{0\}\subset {\mathbb T}_{{V},\xi}$ (where $d_{\xi}\beta$ denotes the differential of $\beta$ at the point $\xi$).
\end{enumerate}}
\end{Def}

\begin{Parrafo}  {\bf Some remarks on conditions (1-3)  in Definition  \ref{def83}.} \cite[\S 8]{Br_V})    
It can be shown that if conditions (1-3)  
hold  at some  point $\xi\in \mbox{Sing}(\Gn)$, then they hold in a neighborhood of $\xi$ in $\mbox{Sing} ({\Gn})$.  Regarding condition  (1),  it can be checked  that  if  $\tau_{{\Gn},{\xi}}\geq e\geq 1$, then  any codimension-$e$-component of ${\Sing}({\Gn})$ containing $\xi$ 
is smooth in a neighborhood of  $\xi$    (cf. \cite[Lemma 13.2]{Br_V}). Therefore this is a canonical center to blow up and a resolution is achieved in one step; hence there is no need to define  an elimination algebra in order to simplify the resolution of $\Gn$. In relation to  condition (2) it is worth noticing that any absolute differential Rees algebra satisfies this condition. Finally,  and regarding condition (3),  it can be shown that almost any smooth local projection defined in an (\'etale)  neighborhood of a  point $\xi \in {\Sing}({\Gn})$ with  $\tau_{{\Gn},{\xi}}\geq e\geq 1$  will satisfy this condition.

\end{Parrafo}

\begin{Def} \label{eliminacion} {\rm Let $\Gn$ be a Rees algebra on a smooth
$n$-dimensional scheme $V^{(n)}$ over a perfect field $k$, and let $\xi\in
\mbox{Sing }{\Gn}$ be a  closed point with $\tau_{{\Gn},{\xi}}\geq e\geq 1$. Let 
$\beta  : V^{(n)}    \to   V^{(n-e)}$ be a $\Gn$-admissible projection in an (\'etale)  neighborhood of $\xi$. Then the ${\Ovne}$-Rees algebra 
$$\Gne:=\Gn\cap {\Ovne}[W],$$
and any other with the same integral closure in ${\Ovne}[W]$,  is an {\em elimination algebra of $\Gn$ in $V^{(n-e)}$}}. 
\end{Def}

\begin{Rem} We underline here that elimination algebras are defined in a different way in \cite{V07} (there, they are defined for $e=1$) and \cite{Br_V} (where the construction is generalized to arbitrary positive integers $e\geq 1$). However, it can be shown, that, up to integral closure,  both definitions lead to the same ${\Ovne}$-Rees algebra (see \cite[Theorem 4.11]{V07}). 

\end{Rem}

\begin{Parrafo} \label{eliminacion_finita} {\bf Local presentations of the multiplicity and elimination algebras.} 
Consider the same  notation and setting as in \ref{PresFinita} for an affine algebraic variey $X=\text{Spec}(B)$ defined over a perfect field $k$ and a point $\xi\in \Mm$.  Recall   that there was a finite morphism $\alpha^*:  X\to V^{(d)}=\text{Spec}(S)$ inducing an    embbeding $X\subset \Vn=\text{Spec}(S[x_1,\ldots, x_{n-d}])$ and a differential  Rees algebra, 
$$\Gn =\G _1^{(d+1)} \odot \ldots \odot \G _{n-d}^{(d+1)}\subset S[x_1,\ldots ,x_{n-d}][W]\mbox{,}$$
which was a  local presentation of the maximum multiplicty of $X$ in a neighborhood of $\xi$. In the following lines we will show that the morphism $\beta: \Vn\to \Vd$ is $\Gn$-admissible and will give a description of  an elimination algebra $\Gd$ of $\Gn$ over  $\Vd$.

\

On the one hand,  it can be checked that for each $i\in \{1,\ldots, (n-d)\}$, the inequality  $\tau_{\G^{(d+1)}_i,\xi}\geq 1$ holds because the $f_i$ are monic polynomials in $x_i$ of degree  $m_i$    defining hypersurfaces of maximum multiplicity $m_i$. In addition, it can be shown that the morphisms  $\beta_{H_i}$ are $\G_i^{(d+1)}$-admissible. Thus, by Definition \ref{eliminacion}, up to integral closure, $\G_i^{(d)}=\G_i^{(d+1)}\cap S[W]$ is an elimination algebra of $\G_i^{(d+1)}$ on $\Vd=\text{Spec}(S)$. 

When the characteristic is zero, up to integral closure, $\G_i^{(d)}$ is the differential Rees algebra generated by the coefficients of the polynomial $f_i\in S[x_i]$ after a Tchirnhausuen transformation. When the characteristic is positive, 
$\G_i^{(d)}$ is generated by suitable symmetric polynomial functions evaluated on the coefficients of the $f_i$ (cf.  \cite{V00}, \cite[\S1, Definition 4.10]{V07}). 

\

Now we claim that $\beta: V^{(n)}\to V^{(d)}$ is $\Gn$-admissible and that, up to integral closure, 
\begin{equation}
\label{amalgama_eliminacion_1}
\G^{(d)}= \G _1^{(d)} \odot \ldots \odot \G _{n-d}^{(d)} \subset S[W]. 
\end{equation}
To prove the claim, first notice that $\tau_{\Gn,\xi} \geq (n-d)$, because the $f_i$ are monic polynomials in $x_i$ of degree  $m_i>1$    defining hypersurfaces of maximum multiplicity $m_i$ in different variables $x_1, \ldots, x_{n-d}$.   Also, since all the $\G_i^{(d+1)}$ are differential Rees algebras, so is $\Gn$. Therefore it can be checked that $\beta: V^{(n)}\to V^{(d)}$ is $\Gn$-admissible and as a consequence, up to integral closure, 
 $$\G _1^{(d)} \odot \ldots \odot \G _{n-d}^{(d)}\subset   \G^{(d)}:=\Gn\cap S[W]  \subset S[W].$$
 
 To show the  equality in (\ref{amalgama_eliminacion_1}) we will use Proposition \ref{finita_elim} below. First, by setting $h=1$ in the proposition it follows that $\G_i^{(d)}\subset {\G^{(d+1)}_i}_{|_{B_i}}$ is a finite extension of $B_i$-Rees algebras for $i=1,\ldots, (n-d)$. Therefore one can conclude that $\G _1^{(d)} \odot \ldots \odot \G _{n-d}^{(d)}\subset \left(\G _1^{(d+1)} \odot \ldots \odot \G _{n-d}^{(d+1)}\right)_{|_B}=\Gn_{|_B}$ is a finite extension of $B$-Rees algebras\footnote{By an abuse of notation, we mean here the extension of $\G _1^{(d)} \odot \ldots \odot \G _{n-d}^{(d)}$ to a $B$-Rees algebra. We will keep on doing this along the rest of the paper.}. Therefore, since $S\subset B$ is finite, $\G _1^{(d)} \odot \ldots \odot \G _{n-d}^{(d)} \subset \Gn_{|_B}\cap S[W]$ is also  a finite extension. 
 
Finally,  by Proposition  \ref{finita_elim},    $\Gd\subset \Gn_{|_B}\cap S[W]$ is   a finite extension of $S$-Rees algebras.  Thus, up to integral closure,  
 $$\G^{(d)}:=\Gn\cap S[W] =\G _1^{(d)} \odot \ldots \odot \G _{n-d}^{(d)} \subset S[W].$$

\end{Parrafo} 

\begin{Prop}\label{finita_elim} \cite[Corollary 7.8]{COA} Let $k$ be a perfect field, let $S$ be a smooth $k$-algebra of dimension $d$. Let $Z_1, \dotsc, Z_h$ denote variables and, for $i = 1, \dotsc, h$, let $f_i(Z_i) \in S[Z_i]$ be a monic polynomial of degree $l_i$. Set
\[
	C
	:= S[Z_1,\cdots,Z_h] / \langle f_1(Z_1), \cdots, f_h(Z_h) \rangle .
\]
Let $\G^{(d+h)}$ be a differential Rees algebra over $S[Z_1,\ldots,Z_h]$ containing $f_1(Z_1)W^{l_1}, \ldots, f_h(Z_h)W^{l_h}$. Then  the natural inclusion $S\subset S[Z_1,\cdots,Z_h]$ is $\G^{(d+h)}$-admissible, and if $\Gd\subset S[W]$ is an elimination algebra of $\G^{(d+h)}$ then  the    inclusion of $C$-Rees algebras,  
\begin{equation} \label{eq:crl:extension_theorem_claim_1}
	\G^{(d)}
	\subset \G_{|_C}^{(d+h)} ,
\end{equation}
is finite. Moreover, as a consequence, there is another inclusion of Rees algebras over $S$, 
\begin{equation} \label{eq:crl:extension_theorem_claim_2}
	\G^{(d)}
	\subset \left (\G^{(d+h)}_{|_C} \cap S[W] \right ),
\end{equation}
which is also finite.

\end{Prop}

\begin{Parrafo}\label{EliminationProperties} {\bf First properties of elimination algebras.} {\rm 
 Let 
$\beta : V^{(n)}    \to   V^{(n-e)}$ be a $\Gn$-admissible projection in an (\'etale)  neighborhood of $\xi\in \Sing (\Gn)$, and let $\Gne\subset \Ovne[W]$ be an elimination algebra.  Then: 

\begin{enumerate}

\item  $\Sing({\mathcal G}^{(n)})$ maps injectively into $\Sing({\mathcal G}^{(n-e)})$, in particular 
$$\beta  (\mbox{Sing}({\mathcal G}^{(n)}))\subset \mbox{Sing}({\mathcal G}^{(n-e)})$$  with  equality if the characteristic is zero, or if ${\mathcal G}^{(n)}$ is a
differential Rees algebra.  Moreover, in this case $\mbox{Sing}({\mathcal G}^{(n)})$ and $\beta  (\mbox{Sing}({\mathcal G}^{(n)}))$ are homeomorphic (see \cite[\S 8.4]{Br_V}).   

\item 
If ${\mathcal G}^{(n)}$ is a differential Rees algebra, then so is    ${\mathcal G}^{(n-e)}$ (see \cite[Corollary 4.14]{V07}).  

\item If ${\mathcal G}^{(n)}\subset {\mathcal G}'^{(n)}$ is a finite extension, then ${\mathcal G}^{(n-e)}\subset {\mathcal G}'^{(n-e)}$ is a finite extension (see \cite[Theorem 4.11]{V07}). 

\item  The order of $\Gne$ at  $\beta (\xi)$ does not depend on the choice of the  projection $\beta$  (see
\cite[Theorem 5.5]{V07} and \cite[Theorem 10.1]{Br_V}).

\item If $\tau_{{\Gn},{\xi}}\geq e+l$ for some non-negative integer $l$, then 
$\tau_{{\Gne},{\beta (\xi)}}\geq l$ (cf., \cite{B}).
\end{enumerate} 
}
\end{Parrafo}

\begin{Rem} \label{resolucion_invariantes} To find a resolution of a given Rees algebra one needs to define invariants at the points of its singular locus,  the most important  being  Hironaka's order function (see Definition \ref{Def:HirOrd}).
However,  this rational number is too coarse and has to be refined. This can be done via elimination algebras which allow us to define Hironaka's order function in lower dimensions as indicated in  the following definition. 

\end{Rem}

\begin{Def} \label{Hironaka_order_e} {\rm Let 
$\beta  : V^{(n)}    \to   V^{(n-e)}$ be a $\Gn$-admissible projection in an (\'etale)  neighborhood of $\xi\in \Sing (\Gn)$, and let $\Gne\subset \Ovne[W]$ be an elimination algebra for some $e\geq 1$. Then, by  \ref{EliminationProperties} (4), 
for $\Gn$ we can define {\em Hironaka's order function in dimension $(n-e)$ at $\xi$} as: 
$$\text{ord}^{(n-e)}_{\Gn}({\xi}): = \text{ord}_{\beta_{(\xi)}}(\Gne).$$
}
\end{Def}

\begin{Rem} \label{nuevo}
With the setting and notation in \ref{eliminacion_finita}, 
recall that $\G^{(d)}= \G _1^{(d)} \odot \ldots \odot \G _{n-d}^{(d)} \subset S[W]$ is an elimination algebra of $\Gn$  (up to integral closure), 
and we have
$$\text{ord}^{(d)}_{\Gn}({\xi})= \text{ord}_{\beta_{(\xi)}}(\G^{(d)})=\mathrm{min}_{i=1,\ldots ,n-d}\left\{ \text{ord}_{\beta_{(\xi)}}(\G_1^{(d)}),\ldots ,\text{ord}_{\beta_{(\xi)}}(\G_{n-d}^{(d)}) \right\} \mbox{.}$$

\end{Rem}

\begin{Parrafo} {\bf Elimination algebras and local sequences. } \label{eliminacion_suclocal} 
Let 
$\beta  : V^{(n)}    \to   V^{(n-e)}$ be a $\Gn$-admissible projection in an (\'etale)  neighborhood of $\xi\in \Sing (\Gn)$, and let $\Gne\subset \Ovne[W]$ be an elimination algebra. Then: 

\begin{enumerate}

\item  The homeomorphism from $\Sing {\mathcal G}^{(n)}$ to 
$\beta (\mbox{Sing
}({\mathcal G}^{(n)}))$  has the following properties: 
If $Z\subset \Sing({\mathcal G}^{(n-e)})$ is a smooth closed subscheme, then 
$\beta^{-1}(Z)_{\text{red}}\cap \Sing({\mathcal G}^{(n)})$ is smooth; and if $Y\subset  \Sing({\mathcal G}^{(n)})$ is a smooth closed subscheme, then so is $\beta(Y)\subset \Sing({\mathcal G}^{(n-e)})$ (\cite[8.4]{Br_V}, \cite[Lemma 1.7]{V00}).  
  
\item  Using (1) it can be shown  that for any  ${\mathcal G}^{(n)}$-local sequence (\ref{Def:GLocSeq}),  there are commutative diagrams  
\begin{equation}
\label{diagrama_compatibilidad}
\xymatrix@R=0pt@C=30pt{
\mathcal{G}^{(n)}=\mathcal{G}_0^{(n)} & \mathcal{G}^{(n)}_1 &  & \mathcal{G}^{(n)}_m \\
V^{(n)}=V^{(n)}_0     \ar[ddd]^\beta     & V^{(n)}_1   \ar[l]_{\rho_0}  \ar[ddd]^{\beta_1}   &  \cdots \ar[l]_{\rho_1} &  V^{(n)}_m \ar[l]_{\rho_{m-1}}  \ar[ddd]^{\beta_m}\\
\\
\\
V^{(n-e)}=V^{(n-e)}_0           & V^{(n-e)}_1   \ar[l]_{\overline{\rho}_0}     & \cdots \ar[l]_{\overline{\rho}_1} &  V^{(n-e)}_m \ar[l]_{\overline{\rho}_{m-1}} \\
\mathcal{G}^{(n-e)}=\mathcal{G}^{(n-e)}_0 & \mathcal{G}^{(n-e)}_1
& & \mathcal{G}^{(n-e)}_m
}
\end{equation}
of transversal projections and  transforms, such that for $i=1,\ldots,m$: 
\begin{enumerate} 
\item If $V_{i-1}^{(n)}\stackrel{\rho_{i-1}}{\longleftarrow} V_i^{(n)}$ is a permissible transformation with center 
$Y_{i-1}\subset \Sing ({\mathcal G}^{(n)}_{i-1}),$ then   $V_{i-1}^{(n-e)}\stackrel{\overline{\rho}_{i-1}}{\longleftarrow} V_i^{(n-e)}$ is the permissible blow up at $\beta_{i-1}(Y_{i-1})$ and $\beta_i: V_i^{(n)} \longrightarrow V_i^{(n-e)}$ is ${\mathcal G}^{(n)}_i$-admissible in an open subset $U_i\subset V_i^{(n)}$ containing $\Sing({\mathcal G}^{(n)}_i)$. 

\item  The Rees algebra ${\mathcal G}^{(n-e)}_i$ is an elimination algebra of ${\mathcal G}_i^{(n)}$  (i.e., the transform of an elimination algebra of  a given Rees algebra ${\mathcal G}^{(n)}$ is the elimination algebra of the transform of ${\mathcal G}^{(n)}$); 
\item  There is an inclusion of closed sets: 
\begin{equation}
\label{Inclusion_Estricta}
\beta_i(\Sing ({\mathcal G}^{(n)}_i))\subseteq \Sing ({\mathcal G}^{(n-e)}_i), 
\end{equation}
and $\Sing ({\mathcal G}^{(n)}_i)$ and $\beta_i(\Sing({\mathcal G}^{(n)}_i))$ are homeomorphic.    If the characteristic is zero then the inclusion 
(\ref{Inclusion_Estricta}) is an equality. 
 \end{enumerate}
See  \cite[Theorem 9.1]{Br_V}. 
\item Conversely, if the characteristic is zero, any ${\mathcal G}^{(n-e)}$-local sequence (\ref{Def:LocalSeq})
induces a ${\mathcal G}^{(n)}$-local sequence and    commutative diagrams 
of transversal projections and  transforms of Rees algebras as in (\ref{diagrama_compatibilidad}) 
satisfying properties (a), (b) and (c) as above. 
\end{enumerate}
\end{Parrafo}

\begin{Parrafo}  {\bf Ress algebras, elimination algebras and resolution.}  
\label{resolution_elimination} 
Consider an $n$-dimensional pair  
$(V^{(n)}, {\mathcal G}^{(n)})$, and let  $\beta: V^{(n)}\longrightarrow V^{(n-e)}$ be some ${\mathcal G}^{(n)}$-admissible projection is fixed in a neighborhood of a point $\xi \in \Sing ({\mathcal G}^{(n)})$ for some $e\geq 1$. 
\begin{enumerate}
\item When the characteristic is zero, it follows from  \ref{eliminacion_suclocal} that a resolution of ${\mathcal G}^{(n)}$ induces a resolution of ${\mathcal G}^{(n-e)}$ and vice-versa: thus finding a resolution of ${\mathcal G}^{(n)}$ is equivalent to finding a resolution of  ${\mathcal G}^{(n-e)}$.
Furthermore, ${\mathcal G}^{(n-e)}$ is the unique $\Ovne$-Rees algebra with this property up to weak equivalence.

\item When the characteristic is positive, the link between ${\mathcal G}^{(n)}$ and ${\mathcal G}^{(n-e)}$ is weaker; however notice that properties (1) and (2) in \ref{eliminacion_suclocal} still hold. In this case it can be shown that $\Gne$ is the largest $\Ovne$-Rees algebra fulfilling  properties (1) and (2). In some sense,  one can think that $\Gne$ is the  $\Ovne$-Rees algebra, that  {\em better approximates} the singular locus of $\Gn$   after considering  $\Gn$-local sequences (see \cite[6.14]{COA}).  
\end{enumerate} 
\end{Parrafo}

\begin{Parrafo}\label{hironaka_order_d}{\rm {\bf Resolutions of Rees algebras vs.  simplifications of the multiplicity.} Let $X$ be a $d$-dimensional variety, and let $(\Vn, \Gn)$ be a local presentation for the multiplicity in an (\'etale)  neighborhoud of a point $\xi\in \underline{\text{Max}} \text{ mult}_X$ as in Definition \ref{Def:LocPres}. As indicated in \ref{Def:LocPres},  a resolution of $\Gn$ induces a sequence of blow ups at equimultiple centers over $X$ that ultimately leads to a simplification of the multiplicity. 

On the other hand, by \ref{resolution_elimination}, when the characteristic is zero, finding a resolution of $\Gn$ is equivalent to finding a resolution of an elimination algebra  in some lower dimensional smooth scheme $\Vne$ (if there is one).  By \cite[Theorem 28.10]{Br_V2}  if $X$ is a variety of dimension $d$ and $(V^{(n)}, \Gn)$ is a local presentation of the multipliticity at some $\xi\in X$, then $\tau_{\Gn, \xi}\geq (n-d)$ and therefore the problem of finding a simplification of the multiplicity of $X$ is equivalent to that of finding a resolution of an elimination algebra of $\Gn$ in dimension $d$. This means that the multiplicity has a {\em local presentation in dimension $d=\text{dim }X$}. 

Furthermore, one can iterate the process of computing elimination algebras in dimensions $(n-1), \ldots, d$ and then  it can be checked that, 
$$1=\text{ord}^{(n)}_{\Gn}({\xi})=\text{ord}^{(n-1)}_{\Gn}({\xi})=\ldots =\text{ord}^{(d+1)}_{\Gn}({\xi})=1\leq \text{ord}^{(d)}_{\Gn}({\xi}),$$
(see Definition \ref{Hironaka_order_e}, \ref{EliminationProperties} (5),  Remark \ref{tau_properties} (1) and Remark \ref{nuevo}). Therefore when facing a simplification of the multiplicity of $X$ at $\xi \in \underline{\text{Max}} \text{ mult}_X  $    the first interesting invariant at $\xi$ is precisely $\text{ord}^{(d)}_{\Gn}({\xi})$ which corresponds to the order of a Rees algebra that represents the multiplicity in dimension $d$. 

\

When the charactersitic is positive, there is still a local presentation of the multiplicity of $X$ at $\xi$, $(\Vn,\Gn)$ (see Theorem \ref{Th:PresFinita}), and  the lower bound  $\tau_{\Gn, \xi}\geq (n-d)$ holds as well (see the discussion in \ref{eliminacion_finita}). One can check as before that 
$$1=\text{ord}^{(n)}_{\Gn}({\xi})=\text{ord}^{(n-1)}_{\Gn}({\xi})=\ldots =\text{ord}^{(d+1)}_{\Gn}({\xi})=1\leq \text{ord}^{(d)}_{\Gn}({\xi}).$$
But here the link between $\Gd$ and $\Gn$ is weaker. In fact, there are examples that show that it is not always  possible to give a local presentation of the multiplicity in dimension $d$ (see \cite[\S 11]{Br_G-E_V}). However, as indicated in \ref{resolution_elimination}, $\Gd$ is the Rees algebra in dimension $d$ that better approximates $\underline{\text{Max}} \text{ mult}_X$ in a neighbourhood of $\xi$ (see \ref{resolution_elimination} (2) above). This means that one way to approach a resolution of $\Gn$ may be by finding a refinement of the invariant  $\text{ord}^{(d)}_{\Gn}$, because the later is too coarse. On the other hand,  it is very natural to ask what the meaning of the rational number  $\text{ord}^{(d)}_{\Gn}({\xi})$ is in this case. It turns out, as we will show in Theorem \ref{Main_Theorem}, that it is related to the rate at which arcs in $X$ with center  $\xi$ separate from  $\underline{\text{Max}} \text{ mult}_X$. More precisely, it is connected to the sequence of Nash multiplicities of the arcs with center $\xi$.  In particular,  this number is intrinsic to $X$ (see Remark \ref{significado}).  

\

To summarize, for a given point $\xi \in \Mm$, and a local presentation of the multiplicity, $(\Vn ,\Gn)$, the invariant $\text{ord}^{(d)}_{\Gn}({\xi})$ (which does not depend on the choice of the $\Gn$-admissible projection) is defined. In addition, it can be shown that  $\text{ord}^{(d)}_{\Gn}({\xi})$ does not depend on the   choice of the local presentation either (see \cite{Br_V2}). Thus, we can  eliminate the reference to $\Gn$ and  define:
\begin{equation}
\label{orden_X_d}
\ord_{\xi}^{(d)}(X):=\text{ord}^{(d)}_{\Gn}({\xi})
\end{equation}
where $(\Vn, \Gn)$ is any local presentation of $\Mm$ in a neighborhood of $\xi$. 
}
\end{Parrafo}

 \section{Jets, arcs, and valuations} \label{Jet_Arcs}

\begin{Def}
Let $Z$ be an arbitrary scheme over a field $k$, and let $K\supset k$ be a field extension.
An {\em m-jet  in   $Z$} is a morphism $\alpha: \text{Spec}K[|t|]/\langle t^{m+1}\rangle\to Z$ for some  $m\in {\mathbb N}$. 
	 \end{Def}
If ${\mathcal S}ch/k$ denotes the category of $k$-schemes and ${\mathcal S}et$ the category of sets, then the contravariant functor:
$$\begin{array}{rcl} 
 {\mathcal S}ch/k & \longrightarrow & {\mathcal S}et\\
 Z  & \mapsto & \text{Hom}_k(Z\times_{\text{Spec}(k)}\text{Spec}(k|[t|]/\langle t^{m+1}\rangle), Z)
\end{array}
$$
is representable by a scheme  ${\mathcal L}_m(Z)$. If $Z$ is of finite type over $k$, then so is  ${\mathcal L}_m(Z)$ (see \cite{Vojta}). For each pair $m\geq m'$ there is the (natural) truncation  map ${\mathcal L}_{m}(Z) \to {\mathcal L}_{m'}(Z)$. In particular, for $m'=0$, ${\mathcal L}_{m'}(Z)=Z$ and we will denote by ${\mathcal L}_m(Z)_{\xi}$ the fiber of the (natural) truncation map over a point $\xi\in Z$. Finaly, if $Z$ is smooth over $k$ then ${\mathcal L}_m(Z)$ is also smooth over $k$ (see \cite{I2}).

By taking the inverse limit of the ${\mathcal L}_m(Z)$,  the {\em arc space of $Z$} is defined,  
$${\mathcal L}(Z):=\lim_{\leftarrow}{\mathcal L}_m(Z).$$ This is the scheme representing the functor 
$$\begin{array}{rcl} 
 {\mathcal S}ch/k & \longrightarrow & {\mathcal S}et\\
 Z  & \mapsto & \text{Hom}_k(Z\tilde{\times}\text{Spf}(k|[t|]), Z).
\end{array}
$$

(see \cite{Bhatt}). 

\begin{Def}\label{arco_def}
	A $K$-point in ${\mathcal L}(Z)$ is called {\em an arc of $Z$} and can be seen as   a morphism $\varphi: \text{Spec}(K[|t|])\to Z$ for some $K\supset k$.  The  image by $\varphi$ of the closed point $\langle 0\rangle$ is called the {\em center of the arc $\varphi$}.  If the center of $\varphi$ is $\xi\in Z$ then it induces a $k$-homomorphism  
	${\mathcal O}_{Z,\xi}\to K[|t|]$ which we will denote by $\varphi$ too; in this case the image by $\varphi$ of the maximal ideal ,  $\varphi({\mathfrak m}_{\xi})$,  generates an ideal $\langle t^m\rangle\subset K[|t|]$ and  then we will say that {\em the order of $\varphi$ is $m$} and will denote it by $\nu_t(\varphi)$.   We will denote  by ${\mathcal L}(Z)_{\xi}$ the set of arcs in  ${\mathcal L}(Z)$ with center $\xi$. The {\em generic point of $\varphi$ in $Z$} is the point in $Z$ determined by the  $\text{kernel}$ of $\varphi$.
\end{Def}

\begin{Def}
We say that an arc 	$\varphi: \text{Spec}(K[|t|])\to Z$ is {\em thin} if it factors through a  proper closed suscheme of $Z$. Otherwise we say that $\varphi$ is {\em fat}. 
	
\end{Def}

\begin{Def} If $Z$ is an  (irreducible)  algebraic variety and  $\alpha: \text{Spec}(K[|t|])\to Z$ is fat, then it defines a discrete valuation  on the quotient field $K(Z)$ of $Z$. This is the {\em valuation corresponding to $\alpha$}. If $\alpha$ is thin, then it defines a valuation in the quotient field $K(Y)$ of some  (irreducible) subvariety $Y\subset Z$. 
\end{Def}

\begin{Def}
Let $\varphi :\mathrm{Spec}(K[[t]])\longrightarrow \mathrm{Spec}(B)$ in $\mathrm{Spec}(B)$ and let $\mathcal{G}=B[g_1W^{b_1},\ldots g_rW^{b_r}]\subset B[W]$ be a $B$-Rees algebra. We define 
$$\varphi({\mathcal{G}}):=K[[t]][\varphi (g_1)W^{b_1},\ldots ,\varphi (g_{n-d})W^{b_{n-d}}]\subset K[[t]][W]\mbox{.}$$
\end{Def}

\begin{Parrafo} \label{entero_arcos}{\bf Integral closure of Rees algebras and arcs.} Let $k$ be a field,  let $B$ be a (not necessary smooth)  reduced $k$-algebra, and let $\G$ be a Rees algebra over $B$. Set $X=\text{Spec}(B)$. For any arc 
$\varphi\in {\mathcal L}_{\infty}(X)$, $\varphi: B\to K[|t|]$, with $k\subset K$ a extension field,   the image via $\varphi$ of $\G$ generates a Rees algebra over $K[|t]]$.  It is clear that, since $\G\subset \overline{\G}$,   the order of the Rees algebra $\varphi(\G)$ at the maximal ideal $\langle t\rangle$, $\ord_t(\varphi(\G))$,   is larger than or equal to  $\ord_t(\varphi(\overline{\G}))$ (here we mean the order as Rees algebras  as in  \ref{Def:HirOrd}). We claim that in fact, 
\begin{equation}
\label{igualdad}
\ord_t(\varphi(\G))=\ord_t(\varphi(\overline{\G})).
\end{equation}
To check the equality, suppose that $\ord_t(\varphi(\G))=s\in {\mathbb Q}$ and  let  $fW^n\in \overline{\G }$.  Then there exist some elements $a_iW^{ni}\in \mathcal{G}$, for $i=1,\ldots ,l$, such that 
\begin{equation}
\label{dependencia}
(fW^n)^l+a_1W^n(fW^n)^{l-1}+\ldots +a_lW^{nl}=0\mbox{.}
\end{equation}
Let $r=\nu_t(\varphi(f))$ be the (usual) order of $\varphi(f)$ at $\langle t\rangle$. We will show that $\frac{r}{n}\geq s$, which will give us the equality in (\ref{igualdad}). 

On the one hand, from  the way in which the coefficients $a_i$ are chosen in (\ref{dependencia}), one has that for $i=1,\ldots, l$, 
\begin{equation}
\label{pesos_coeficientes}
\frac{\nu _t(\varphi (a_i))}{ni}\geq s\mbox{.}
\end{equation}
On the other, from equation (\ref{dependencia}) it follows that there must be an index $i\in \{1,\ldots, l\}$ such that 
\begin{equation}
\label{eq:ord_resta1}
\nu_t(\varphi(a_i(f)^{l-i}))=rl. 
\end{equation}

Now suppose, contrary to our claim, that $\frac{r}{n}<s$. Then, by (\ref{pesos_coeficientes}), for $i=1,\ldots ,l$,
$$\nu _t(\varphi(a_i(f)^{l-i}))\geq sni+r(l-i)=rl+i(sn-r)>rl\mbox{,}$$
which contradicts (\ref{eq:ord_resta1}). 
\end{Parrafo}

\section{Nash multiplicity sequences, persistance, and the algebra of contact}\label{RAandNash}

 \subsection*{Nash multiplicity sequences}
 
Let $X$ be an algebraic variety defined over a   perfect field $k$    and  let  $\xi \in \Mm $ be a (closed) point.   Assume that $X$ is locally  a hypersurface in a neighborhood of $\xi$,  $X\subset V$, where   $V$ is smooth over $k$   and work at the completion $\widehat{\mathcal O}_{V,{\mathfrak m}_{\xi}}$. Under these hypotheses,  in \cite{L-J}, Lejeune-Jalabert introduced the   {\em Nash multiplicity sequence  along an arc $\varphi\in {\mathcal L}(X)_{\xi}$}.  This is a non-increasing sequence of non-negative integers 
\begin{equation}
\label{Nash_sequence}
m_0\geq m_1\geq \ldots \geq m_l=m_{l+1}=...\geq 1, 
\end{equation}
where  $m_0$ is the usual multiplicity of $X$ at $\xi$, and the rest of the terms  are computed by considering  suitable stratifications on ${\mathcal L}_{m}(X)_{\xi}$  defined via the action of certain differential operators on the fiber of the jets spaces  ${\mathcal L}_{m}(\text{Spec}(\widehat{\mathcal O}_{V,{\mathfrak m}_{\xi}}))$ over $\xi$ for $m\in {\mathbb N}$.  The sequence (\ref{Nash_sequence})   can be interpreted,  in some sense,  as  the {\em multiplicity of $X$ along the arc $\varphi$}: thus it can be seen as a refinement of the usual multiplicity.   The sequence stabilizes at the value given by the multiplicity $m_l$ of  $X$ at the generic point of the arc $\varphi $ in $X$ (see \cite[\S 2, Theorem 5]{L-J}).  

\vspace{0.2cm}

In \cite{Hickel05}, Hickel generalized Lejeune's construction  to the case of an arbitrary variety $X$ and presented   the sequence (\ref{Nash_sequence}) in a different way  which we will explain along the following lines.

\vspace{0.2cm}

Since the  arguments are of local nature, let us suppose that  $X=\text{Spec}(B)$  is a $d$-dimensional variety defined over a perfect field $k$. Let $\xi \in \Mm $ be a point (which we may assume to be closed) of multiplicity $m=\mm $, and let $\varphi $ be an arc in $X$ centered at $\xi $. Consider   the natural morphism   
$$\Gamma _0=\varphi \otimes i:B\otimes_k k[t]\rightarrow K[[t]]\mbox{,}$$
which is additionally an arc in $X_0=X\times \mathbb{A}^1_k$ centered at the point $\xi _0=(\xi ,0)\in X_0$. These elements determine completely a sequence of blow ups at points:
\begin{equation}\label{intro:diag:Nms}
\xymatrix@R=15pt@C=50pt{
\mathrm{Spec}(K[[t]]) \ar[dd]^{\Gamma _0} \ar[ddr]^{\Gamma _1} \ar[ddrrr]^{\Gamma _l} & & & & \\
& & & & \\
X_0=X\times \mathbb{A}^1_k & X_1 \ar[l]_>>>>>{\pi _1} & \ldots \ar[l]_{\pi _2} & X_l \ar[l]_{\pi _l} & \ldots \\
\xi _0=(\xi ,0) & \xi _1 & \ldots & \xi _l & \ldots 
}\end{equation}
Here, $\pi _i $ is the blow up of $X_{i-1}$ at $\xi _{i-1}$, where $\xi _{i}=\mathrm{Im}(\Gamma _{i})\cap \pi _{i}^{-1}(\xi _{i-1})$ for $i=1,\ldots ,l,\ldots $, and $\Gamma _i$ is the (unique) arc in $X_i$ with center $\xi _i$ which is obtained by lifting $\Gamma _0$ via the proper morphism $\pi _i\circ \ldots \circ \pi _1$.     This sequence of blow ups defines a non-increasing sequence
\begin{equation}
\label{Nash_sequence_2}
m_0\geq m_1\geq \ldots \geq m_l=m_{l+1}=...\geq 1, 
\end{equation}
where $m_i$ corresponds to the multiplicity of $X_i$ at $\xi _i$ for each $i=0,\ldots ,l,\ldots $. Note that $m_0$ is nothing but the multiplicity of $X$ at $\xi $, and it is proven that for hypersurfaces  the sequence (\ref{Nash_sequence_2}) coincides with the sequence (\ref{Nash_sequence}) above.  We will refer to the sequence of blow ups in (\ref{intro:diag:Nms}) as the {\em sequence of blow ups directed by $\varphi$}.

\begin{Rem}\label{Hickel_construction}
	Using Hickel's construction, it can be checked that the  first index $i\in \{1,\ldots, l+1\}$ for which there is a strict inequality   in (\ref{Nash_sequence_2})  (i.e., the first index $i$ for which  $m_0>m_i$) can be interpreted as the minimum   number of steps needed to {\em separate  the graph of $\varphi$ from } $\text{\underline{Max} mult}_{X_0}$   by blow ups  \footnote{Actually, to be precise, this statement has to be interpreted  in $B\otimes K[|t|]$, where the graph of $\varphi$ is defined.}. This will necessarily be a finite number as long as the generic point of $\varphi$ is not contained in $\Mm$.
	\end{Rem}

\subsection*{The persistance and its link to Hironaka's order function}

Let $X$ be an algebraic variety defined over a perfect field $k$ and let $\xi\in \Mm$ be a point of multiplicity $m$. Let $\varphi\in {\mathcal L}(X)_{\xi}$, and consider, as in (\ref{Nash_sequence_2}), the Nash multiplicity sequence along $\varphi $. For the purposes of this paper, we will pay attention to the first time that the Nash multiplicity drops below $m$ (see Remark \ref{Hickel_construction} above). The contents of this subsection where in part  developed in \cite{Br_E_P-E}, but we include the whole argument here to facilitate the reading of the paper.

\begin{Def}\label{def:rho}
	Let $\varphi $ be an arc in $X$ with center $\xi \in \Mm$ 
		 whose generic point is not contained in $\Mm$. We denote by $\rho _{X,\varphi }$ the minimum number of blow ups directed by $\varphi $ which are needed to lower the  Nash multiplicity of  $X$ at $\xi $. That is, $\rho _{X,\varphi }$ is such that $m=m_0=\ldots =m_{\rho _{X,\varphi }-1}>m_{\rho _{X,\varphi }}$ in the sequence (\ref{Nash_sequence_2}) above. We call $\rho _{X,\varphi }$ the \textit{persistance of $\varphi $ in  $\Mm $}. We denote by $\rho _{X}(\xi )$ the infimum of the number of blow ups directed by some arc in $X$ through $\xi $ needed to lower the Nash multiplicity   at $\xi$:
	\begin{align*}
	\rho _X: \mathrm{\underline{Max}\; mult}_X & \longrightarrow \mathbb{N}\\
	\xi & \longmapsto \rho _X(\xi )=\inf _{\varphi \in \mathcal{L}(X)_{\xi }}\left\{ \rho _{X,\varphi }\right\} \mbox{.}
	\end{align*}
\end{Def}

To keep the notation as simple as possible, $\rho _{X,\varphi }$ does not contain a reference to the point $\xi $, even though it is clear that it is local. However, the point is determined by $\varphi $, and hence it is implicit, although not explicit in the notation. Similarly, we may refer to $\rho _X(\xi )$ as $\rho _X$ once the point is fixed.\\
\\
Let us define normalized versions of $\rho _{X,\varphi }$ and $\rho _X$ in order to avoid the influence of the order of the arc in the number of blow ups needed to lower the  Nash multiplicity. 

\begin{Def}\label{def:rho_bar}  
	For a given arc $\varphi: \text{Spec}(K[|t|])\to X$   with center $\xi\in \Mm$, we will write
	$$\bar{\rho }_{X,\varphi }=\frac{\rho _{X,\varphi }}{\nu_t(\varphi )}\mbox{, } \ \mbox{ and } \ \bar{\rho }_{X}(\xi )=\inf _{\varphi \in \mathcal{L}(X)_{\xi }}\left\{ \bar{\rho }_{X,\varphi }\right\} \mbox{,}$$
	where $\nu_t(\varphi)$ denotes the oder of the arc, i.e., the usual order of $\varphi({\mathfrak m}_{\xi})$ at $K[|t|]$. 
\end{Def}
%\vspace{1cm}

\begin{Rem} \label{significado} 
	As we will see in Section \ref{Nash_Hironaka} (see (\ref{limite_ord_d})), the value at a point $\xi\in \Mm$ of Hironaka's order function in dimension $d$ (see \ref{hironaka_order_d}) can be read from the numbers in  Definition \ref{def:rho_bar} above. In fact, the expression (\ref{limite_ord_d})  gives an intrinsic definition of this rational number and provides at the same time a geometrical meaning for it  (see Remark \ref{Hickel_construction}). 
\end{Rem}

\subsection*{The algebra of contact and the order of contact}

In the present  section, we will show that for $X$, $\xi \in \Mm$ and $\varphi \in \mathcal{L}(X)_{\xi} $, we can attach a Rees algebra to the sequence of blow ups directed by $\varphi $ (see (\ref{intro:diag:Nms})). From this algebra, we will define a new quantity, $r_{X,\varphi }$ (see Definition \ref{def:r}), which is a refinement of $\rho _{X,\varphi }$. In particular, $\rho _{X,\varphi }$ is obtained by taking the integral part of $r_{X,\varphi }$ (see Proposition \ref{thm:r_elim_amalgama}).

To define $r_{X,\varphi }$, we need to introduce the \textit{algebra of contact} of $\varphi $ with $\Mm $. This was carefully developed in \cite[Section 4]{Br_E_P-E} for varieties defined over fields of characteristic zero.   However, all of the contents of that section are also valid over perfect fields of arbitrary characteristic. We refresh here the notation used there,  and refer to the results which are characteristic free. 

\vspace{0.2cm}

\begin{Parrafo}\label{notacion_setting}  {\bf Notation and setting.}  Recall that, locally, in an (\'etale) neighborhood\footnote{which we will also denote by $X$.} of $\xi \in \Mm$, it is possible to find an immersion $X\hookrightarrow V^{(n)}$ and an $\mathcal{O}_{V^{(n)},\xi }$-Rees algebra   $\G^{(n)}$, which we may assume to be differentially closed, representing the multiplicity of $X$. That is, such that $\Sing (\G _X^{(n)})=\Mm $, and so that this condition is preserved by $\G _X^{(n)}$-local sequences over $V^{(n)}$ as long as the maximum multiplicity does not decrease (see   Theorem \ref{Th:PresFinita} and the discussion in  \ref{PresFinita}). Consider $X_0=X\times \mathbb{A}_k^1$ as in (\ref{intro:diag:Nms}). After the product by $\mathbb{A}_k^1$, there is also an immersion $X_0\hookrightarrow V^{(n)}\times \mathbb{A}_k^1=V_0^{(n+1)}$, and  $\G^{(n)}$,   can be extended
to   the   smallest Rees algebra $\G^{(n+1)}_0$ over $\mathcal{O}_{V_0^{(n+1)}}$ containing $\G^{(n)}$, which moreover represents the multiplicity of $X_0$ locally in an (\'etale) neighborhood of $\xi _0=(\xi ,0)$.    Notice that $\G^{(n+1)}_0$ is  also differentially closed.  

\vspace{0.2cm}

The sequence of blow ups (\ref{intro:diag:Nms}) directed by $\varphi $ induces also a sequence of point blow ups for $V_0^{(n+1)}$:

\small
 \begin{equation}\label{diag:perm_seq1c}
\xymatrix{ 
(V_0^{(n+1)},\xi _0) & & (V_1^{(n+1)},\xi _1) \ar[ll]_{\pi _1} & & \ldots \ar[ll]_>>>>>>>>>>{\pi _2} & & (V_r^{(n+1)},\xi _r) \ar[ll]_<<<<<<<<<{\pi _r} \\
(X_0^{(d+1)},\xi _0) \ar@{^{(}->}[u] & & (X_1^{(d+1)},\xi _1) \ar[ll]_{\left. \pi _1\right| _{X_1^{(d+1)}}} \ar@{^{(}->}[u] & & \ldots \ar[ll]_>>>>>>>>>>{\left. \pi _2\right| _{X_2^{(d+1)}}} & & (X_r^{(d+1)},\xi _r) \ar[ll]_<<<<<<<<<{\left. \pi _r\right| _{X_r^{(d+1)}}} \ar@{^{(}->}[u] \\
(\mathrm{Spec}(K[[t]]),0) \ar[u]^{\Gamma _0} & & (\mathrm{Spec}(K[[t]]),0) \ar[ll]_{id} \ar[u]^{\Gamma _1} & & \ldots \ar[ll]_>>>>>>>>>>>>{id} & & (\mathrm{Spec}(K[[t]]),0)\mbox{.} \ar[ll]_<<<<<<<<<<<{id} \ar[u]^{\Gamma _r} \\
}
\end{equation}
\normalsize

Observe that the arc $\Gamma_0$ naturally induces  another   arc,   the graph of $\varphi$,   
\begin{equation}\label{gamma0}
	\tilde{\Gamma }_0=\varphi\otimes_k \text{Id}: (\mathcal{O}_{X,{\xi}}\otimes _{k} K[[t]])_{{\tilde{\xi}_0}}\longrightarrow K[[t]], 
	\end{equation}
(where $\tilde{\xi}_0$ denotes the point $(\xi,0)$ in $\text{Spec} (\mathcal{O}_{X,{\xi}}\otimes _{k} K[[t]])$)   and also  a commutative diagram,  
\begin{equation}\label{diag:triangle}
\xymatrix@C=0.5pc@R=1pc{\mathcal{O}_{V^{(n)},{\xi} } \ar[d] \ar[rrr]  & & & \mathcal{O}_{V_0^{(n+1)},{\xi _0}}  \ar[rrrrrrr]  \ar[d] & & & & & & & (\mathcal{O}_{V^{(n)},{\xi} }\otimes _{k} K[[t]])_{\tilde{\xi _0}} \ar[d] \ar[rrrrd]^{\tilde{\Delta}_0} &  & & &  \\
	\mathcal{O}_{X^{(d)},\xi } 
	%\ar[ddddrrr]_{\varphi }  
	\ar[rrr] & & & \mathcal{O}_{X_0^{(d+1)},{\xi _0}} \ar[rrrrrrr] %\ar[dddd]^{\Gamma _0 }
	 & & & & & & & (\mathcal{O}_{X^{(d)},{\xi}}\otimes _{k} K[[t]])_{\tilde{\xi _0}}
	 \ar[rrrr]^{\tilde{\Gamma}_0} & & &  & K[[t]].  
	  %\ar@/^3pc/[ddddlllllll]^{\tilde{\Gamma }_0}
	 %\\
	%\overline{y_i} \ar@{|->}[dddrr] & & & & \overline{y_i},t \ar@{|->}[dd] \ar@{|->}[rrrr] & & & & \overline{y_i},t \ar@/^1pc/@{|->}[ddllll] & & 
	%\\
	%& & & & & & & & & & \\
	%& & & & \varphi _{y_i},t & & & & & & \\
%	& & \varphi _{y_i}=\varphi (y_i) & K[[t]] & & & & & & & \\
}
\end{equation}

Now set,
\begin{equation}
\label{def_v}
\tilde{V}_0^{(n+1)}=\mathrm{Spec}(\mathcal{O}_{V^{(n)},{\xi} }\otimes _{k} K[[t]])_{\tilde{\xi _0}} \ \text{ 
and  } \ \tilde{X}_0^{(d+1)}=\mathrm{Spec}(\mathcal{O}_{X^{(d)},{\xi}}\otimes _{k} K[[t]])_{\tilde{\xi _0}}\mbox{,}
\end{equation}

and let $C_0\subset \tilde{X}_0\subset \tilde{V}_0^{(n+1)}$ be the     regular   curve defined by $\tilde{\Gamma }_0$, that is, the closure of the generic point of the arc $\tilde{\Gamma }_0$.   Let  $y_1,\ldots, y_n$ be a regular system of parameters at ${\mathcal O}_{V^{(n)},{\xi}}$. Their images at  ${\mathcal O}_{\tilde{V}_0^{(n+1)},\tilde{\xi}_0}$,  say $\tilde{y}_1,\ldots, \tilde{y}_n$, are part of a regular system of parameters    ${\mathcal O}_{\tilde{V}_0^{(n+1)},\tilde{\xi}_0}$, and moreover,  
\begin{equation}
\label{parameters}
\langle  \tilde{y}_1,\ldots, \tilde{y}_n, t \rangle ={\mathfrak m}_{\tilde{\xi}_0} \subset {\mathcal O}_{\tilde{V}_0^{(n+1)},\tilde{\xi}_0}.
\end{equation}
 Set $h_i=\tilde{y}_i-\varphi({\tilde{y}_i})\in {\mathcal O}_{\tilde{V}^{(n+1)}_0}$ for $i=1,\ldots, n$. Then $C_0$ is (the    regular   curve) defined  in $\tilde{V}_0^{(n+1)}$ by the ideal 
\begin{equation}\label{eq:h_i_arc}
\langle h_1,\ldots , h_n\rangle\mbox{.}\end{equation}  
Thus the arc $\varphi $ naturally induces also a sequence of blow ups at points for $\tilde{V}_0^{(n+1)}$ and $C_0$:
\begin{equation}\label{diag:perm_seq2}
\xymatrix{ 
	(\tilde{V}_0^{(n+1)},\xi _0) & & (\tilde{V}_1^{(n+1)},\xi _1) \ar[ll]_{\tilde{\pi }_1} & & \ldots \ar[ll]_>>>>>>>>>>{\tilde{\pi }_2} & & (\tilde{V}_r^{(n+1)},\xi _r) \ar[ll]_<<<<<<<<<{\tilde{\pi }_r} \\
	(\tilde{X}_0^{(d+1)},\xi _0) \ar@{^{(}->}[u] & & (\tilde{X}_1^{(d+1)},\xi _1) \ar[ll]_{\left. \tilde{\pi }_1\right| _{\tilde{X}_1^{(d+1)}}} \ar@{^{(}->}[u] & & \ldots \ar[ll]_>>>>>>>>>>{\left. \tilde{\pi }_2\right| _{\tilde{X}_2^{(d+1)}}} & & (\tilde{X}_r^{(d+1)},\xi _r) \ar[ll]_<<<<<<<<<{\left. \tilde{\pi }_r\right| _{\tilde{X}_r^{(d+1)}}} \ar@{^{(}->}[u] \\
	(C_0,\xi _0) \ar@{^{(}->}[u] & & (C_1,\xi _1) \ar[ll]_{\left. \tilde{\pi }_1\right| _{C_1}} \ar@{^{(}->}[u] & & \ldots \ar[ll]_>>>>>>>>>>>>{\left. \tilde{\pi }_2\right| _{C_2}} & & (C_r,\xi _r) \ar[ll]_<<<<<<<<<<<{\left. \tilde{\pi }_r\right| _{C_r}}, \ar@{^{(}->}[u] 
}
\end{equation}
where $C_i$ denotes the strict transform of $C_{i-1}$ for $i=1,\ldots, r$. Finally, we define the Rees algebra 
$${\mathcal C}_0:={\mathcal O}_{\tilde{V}_0^{(n+1)},\tilde{\xi}_0}[h_1W, \ldots, h_nW]$$
 on $\tilde{V}_0^{(n+1)}$, so that $\mathrm{Sing}({\mathcal C}_0)=C_0$. 
Observe that for any ${\mathcal C}_0$-local sequence over  $\tilde{V}_0^{(n+1)}$  in the sense of Definition \ref{Def:GLocSeq}\footnote{Although Definition \ref{Def:GLocSeq} is stated for smooth schemes, it is equaly valid for regular schemes.}, 
$$\xymatrix@R=2pt@C=30pt{(\tilde{V}_0^{(n+1)}, {\mathcal C}_0) & \ar[l] (\tilde{V}_1^{(n+1)}, {\mathcal C}_1) &  \ar[l] \ldots & \ar[l] (\tilde{V}_s^{(n+1)}, {\mathcal C}_s) \\   &  & &  \\ 
	C_0 \ar@{^{(}->}[uu] & 	\ar[l] C_1 \ar@{^{(}->}[uu] &  \ar[l] \ldots & 	\ar[l]C_s \ar@{^{(}->}[uu]
}$$
one has that $\Sing({\mathcal C}_i)=C_i$, where $C_i$ is the strict transform of $C_0$ in $\tilde{V}_i^{(n+1)}$ for $i=1,\ldots, s$. 
\end{Parrafo}

\begin{Def} \label{def:AlgContact}
Consider the same notation and  setting as in \ref{notacion_setting}. 
By an \textit{algebra of contact of $\varphi $ with $\Mm $}   on $\tilde{V}_0^{(n+1)}$,   we mean an $\mathcal{O}_{C_0}$-Rees algebra ${\mathcal H}$ such that 
\begin{equation}\label{form:intersection}
 \Sing ({\mathcal H})= C_0 \cap \left\{ \eta \in \tilde{X}_0: \mathrm{mult}_{\eta }(\tilde{X}_0)=m\right\} = \Sing ({\mathcal C}_0)\cap \Sing ({\mathcal G}^{(n+1)}_0)\subset C_0\mbox{,}
\end{equation}
and such that for any local sequence on $\tilde{V}_0^{(n+1)}$ that is both $\G^{(n+1)}_0$-local and ${\mathcal C}_0$-local, 
$$\xymatrix@R=2pt@C=30pt{\G^{(n+1)}, {\mathcal C}_0  & \G^{(n+1)}_{1}, {\mathcal C}_1  & & \G^{(n+1)}_{s},  {\mathcal C}_s\\ 
\tilde{V}_0^{(n+1)}  & \ar[l] \tilde{V}_1^{(n+1)} &  \ar[l] \ldots & \ar[l]  \tilde{V}_s^{(n+1)}\\  &  & &  \\ 
		 C_0 \ar@{^{(}->}[uu]  & 	  \ar[l] C_1  \ar@{^{(}->}[uu] &   \ar[l] \ldots    & 		 \ar[l] C_s \ar@{^{(}->}[uu] }$$
one has that  
$$\Sing ({\mathcal H} _i)=C_i \cap \left\{ \eta \in \tilde{X}_i^{(d+1)}: \mathrm{mult}_{\eta }(\tilde{X}_i^{(d+1)})=m\right\}=\Sing ({\mathcal C}_i)\cap \Sing (\G^{(n+1)}_i) \subset C_i$$ for $i=1,\ldots ,s$. 

\end{Def}

\begin{Rem} From the previous definition it follows that: 
\begin{enumerate}
	\item[(i)] Lowering the Nash multiplicity along an arc $\varphi $ in $X$ at $\xi \in \Mm $ below $m=\mm $, is   equivalent to resolving the Rees algebra ${\mathcal H}$, and consequently $\rho _{X,\varphi }$ as in Definition \ref{def:rho} is the number of induced transformations by (\ref{intro:diag:Nms}) of this Rees algebra ${\mathcal H}$ which are necessary to resolve it (see Definition \ref{def:res_RA}).
	\item[(ii)] From the way in which it has been defined, the algebra of contact of $\varphi $ with $\Mm $, if it exists, is unique up to weak equivalence. Therefore, the order of {\em any} algebra of contact of $\varphi$ with $\Mm$ at $\tilde{\xi}_0$ is the same (this follows from Hironaka's Trick \cite[7.1]{E_V97}). This motivates the following definition. 
\end{enumerate}	
\end{Rem}

\begin{Def}\label{def:r}
	Let $X$ be a variety, and let $\varphi $ be an arc in $X$ through $\xi \in \Mm $  as in  \ref{notacion_setting}.  We define the \textit{order of contact} of $\varphi $ with $\Mm $ as the order\footnote{As we have done before, we will write $\xi $ for the image of $\xi $ under most of the morphisms we use, as long as the identification between both points is clear.} at $\xi $ of   any algebra of contact of $\varphi$ with $\Mm$, and denote it by 
	$r_{X,\varphi }$.  
	Normalizing $r_{X,\varphi }$ by the order of the arc (see Definition \ref{arco_def}) we define:
	\begin{equation} \label{formula_r}
	\bar{r}_{X,\varphi }=\frac{r_{X,\varphi }}{\nu_t(\varphi )}\in \mathbb{Q}\mbox{.}
	\end{equation}
	Let us denote 
	\begin{equation}\label{eq:set_r}
	\Phi _{X,\xi }=\left\{ \overline{r}_{X,\varphi } \right\} _{\varphi }\subset \mathbb{Q}_{\geq 1}\mbox{,}
	\end{equation}
	where $\varphi $ runs over all arcs in $X$ with center $\xi $ whose generic point is not contained in $\Mm $.
\end{Def}

The next result guarantees the existence of algebras of contact: 

\begin{Prop}\label{prop:eq_amalgama_ordenes} 
	Let $X$ be a variety defined over a perfect field $k$, let $\xi $ be a point in $\Mm $, and let $\varphi $ be an arc in $X$ through $\xi $ with the hypotheses and notation in \ref{notacion_setting}. Then 
	 the restriction of the differential Rees algebra ${\mathcal G}^{(n+1)}_0$ to ${\mathcal O}_{C_0}$ is an algebra of contact  of $\varphi $ with $\Mm $. 
\end{Prop}

\begin{proof}

We use the notation of \ref{notacion_setting} and the line of argument used in   \cite[Proposition 4.4]{Br_E_P-E}. 

By construction $C_0\cong\Spec(K[[t]])$ via the arc $\tilde{\Delta}_0$ (\ref{diag:triangle}).
On the other hand, from the definition of $\tilde{V}_{0}^{(n+1)}$ (see (\ref{def_v}))  we have that the natural morphism 
$\tilde{V}_{0}^{(n+1)}\to\Spec(K[[t]])$ is smooth.
So that there is a smooth retraction $$\sigma:\tilde{V}_{0}^{(n+1)}\to C_0.$$
Denote by $i:C_0\to \tilde{V}_{0}^{(n+1)}$   the inclusion morphism.
The restriction of ${\mathcal G}^{(n+1)}_0$ is the pull back $i^{\ast}({\mathcal G}^{(n+1)}_0)$ in $\mathcal{O}_{C_0}$.

Now set
$$\mathcal{H}^{(n+1)}={\mathcal G}^{(n+1)}_0\odot \mathcal{C}_0.$$
Note that $\Sing(\mathcal{H}^{(n+1)})\subset C_0$ and this inclusion is stable by any local sequence.
This means that the algebra $\mathcal{H}^{(n+1)}\cap \mathcal{O}_{C_0}[W]$ is an algebra of contact of $\varphi$, according to Definition~\ref{def:AlgContact}.

Finally, since $\mathcal{G}^{(n)}$ is a differential Rees algebra, it can be  checked, at the completion of the regular local ring ${\mathcal O}_{\tilde{V}_{0}^{(n+1)}}$,
that
 %por paso al completado y inducción en b (peso en W^b)
\begin{equation}
\label{descomposicion_retraccion}
\mathcal{H}^{(n+1)}=\sigma^{\ast}\left(i^{\ast}\left({\mathcal G}^{(n+1)}_0\right)\right)\odot \mathcal{C}_0
\end{equation}
from where the result follows.

\end{proof}

\begin{Rem} \label{justificacion} In the following lines we explain the meaning of   Proposition \ref{prop:eq_amalgama_ordenes} and give an explicit expression to compute the order of contact.    With the same notation and setting as in \ref{notacion_setting}, suppose $\Gn =\mathcal{O}_{V^{(n)},\xi }[g_1W^{c_1},\ldots ,g_{s}W^{c_s}]$ is a differential Rees algebra representing the multiplicity of $X$ locally in an (\'etale) neighborhood of $\xi $ in  $\Vn$. Note that  ${\G^{(n+1)}_0}$ is nothing but the extension  
	of $\Gn$ to $\tilde{V}_0^{(n+1)}=\mathrm{Spec}(\mathcal{O}_{V^{(n)},{\xi} }\otimes _{k} K[[t]])_{\tilde{\xi _0}}$. Acording to Proposition \ref{prop:eq_amalgama_ordenes}, the restriction of  ${\G^{(n+1)}_0}$ to $C_0$ is an algebra of contact. 	
	Now, going back to diagram (\ref{diag:triangle}), we get another commutative diagram, 
	\begin{equation}\label{diag:triangle_1}
	\xymatrix@C=0.5pc@R=1pc{\mathcal{O}_{V^{(n)},{\xi} } \ar[d] \ar[rrr]  & & & \mathcal{O}_{V_0^{(n+1)},{\xi _0}}  \ar[rrrrrrr]  \ar[d] & & & & & & & (\mathcal{O}_{V^{(n)},{\xi} }\otimes _{k} K[[t]])_{\tilde{\xi _0}} \ar[d] \ar[rrrrd]^{\tilde{\Delta}_0} &  & & &  \\
		\mathcal{O}_{X^{(d)},\xi } 
		%\ar[ddddrrr]_{\varphi }  
		\ar[rrr] & & & \mathcal{O}_{X_0^{(d+1)},{\xi _0}} \ar[rrrrrrr] %\ar[dddd]^{\Gamma _0 }
		& & & & & & & (\mathcal{O}_{X^{(d)},{\xi}}\otimes _{k} K[[t]])_{\tilde{\xi _0}} \ar[d]
		\ar[rrrr]^{\tilde{\Gamma}_0} & & &  & K[[t]] \\ 
		%\ar[ddddrrr]_{\varphi }  
	 & & &  %\ar[dddd]^{\Gamma _0 }
		& & & & & & & (\mathcal{O}_{{C_0},{\tilde{\xi}_0}}),
		\ar[rrrru]_{\Psi_0} & & &  &   
		}
	\end{equation}
because the arc $\tilde{\Delta}_0$ (induced by $\Gamma_0$  defined in (\ref{gamma0}))	factorizes through ${\mathcal O}_{C_0}$  (see (\ref{eq:h_i_arc}).  
Now the restriction of ${\mathcal G}^{(n+1)}_0$ to $\mathcal{O}_{{C_0},{\tilde{\xi}_0}}$ is just the image of $\Gn$ in $\mathcal{O}_{{C_0},{\tilde{\xi}_0}}[W]$. 

On the other hand,   the image of the maximal ideal in 
$\mathcal{O}_{{C_0},{\tilde{\xi}_0}}$ via $\Psi_0$ is $\langle t\rangle\subset K[|t|]$ (see (\ref{parameters}) and (\ref{eq:h_i_arc})).  Therefore, the order of the image of $\Gn$ in $\mathcal{O}_{{C_0},{\tilde{\xi}_0}}[W]$ (i.e., the order of he algebra of contact at $\tilde{\xi}_0\in C_0$)  is the same as the order at $\langle t\rangle$ of $\tilde{\Delta}_0(\G_0^{(n+1)})=\varphi(\Gn)\subset K[|t|][W]$  (see (\ref{descomposicion_retraccion})). 
As a consequence,  the order of contact of $\varphi$ with $\Mm$ can be rewritten   as:  
\begin{equation}
\label{r_escrito}
r_{X,\varphi }=\mathrm{ord}_{t}(\varphi (\G ))\in \mathbb{Q}\mbox{.}
\end{equation}
And the normalized version of (\ref{formula_r}) is:   
\begin{equation}
\bar{r}_{X,\varphi }=\frac{\mathrm{ord}_{t}(\varphi (\G ))}{\nu_t(\varphi )}\in \mathbb{Q}\mbox{.}
\end{equation}

\end{Rem}

\begin{Prop}\label{thm:r_elim_amalgama}\cite[Proposition 4.11]{Br_E_P-E} 
Let $X$ be a variety defined over a perfect field $k$,  let $\xi $ be a point in $\Mm $ and let $\varphi $ be an arc in $X$ through $\xi $. Then
\begin{equation}
\rho _{X,\varphi }=\lfloor r_{X,\varphi }\rfloor \mbox{.}
\end{equation}
That is, the persistance of $\varphi $ in $X$ equals the integral part of the order of contact of $\varphi $ with $\Mm $.
\end{Prop}

\vspace{0.2cm}

\section{Nash multiplicity sequences and Hironaka's order function} \label{Nash_Hironaka}

\vspace{0.2cm}

The results obtained in \cite{Br_E_P-E} showed that, for varieties defined over fields of characteristic zero, the invariant $\ord _{\xi }^{(d)}(X)$ at a point $\xi \in \Mm $ can be read in the space of arcs of $X$. More precisely: given   $\varphi :\mathrm{Spec}(K[[t]])\longrightarrow X$, centered at $\xi $, one can consider the family of arcs given as $\varphi _n=\varphi \circ i_n$ for $i>1$, where $i_n^*:K[[t]]\longrightarrow K[[t^n]]$ maps $t$ to $t^n$. Then:
$$\bar{r}_{X,\varphi }=\frac{1}{\nu_t(\varphi )}\cdot \lim _{n\rightarrow \infty }\frac{\rho _{X,\varphi _{n}}}{n}\mbox{,}$$
and hence
\begin{equation}
\label{limite_ord_d}
\mathrm{ord}_{\xi }^{(d)}(X)=\mathrm{inf}_{\varphi }\left( \frac{1}{\nu_t(\varphi )}\cdot \lim _{n\rightarrow \infty }\frac{\rho _{X,\varphi _{n}}}{n} \right) \mbox{,} 
\end{equation} 
where $\varphi $ runs over all arcs in $X$ centered at $\xi $ which are not contained in $\Mm $, and the infimum is, in fact, a minimum (see Definition \ref{def:r} and Remark \ref{justificacion}). This is a consequence of the following Theorem:

\begin{Thm}  \label{Main_Theorem}
Let $X$ be an algebraic variety of dimension $d$ defined over a perfect field $k$, and let $\xi $ be a point in $\Mm $. Then:
$$\mathrm{inf} \Phi _{X,\xi }=\mathrm{min} \Phi _{X,\xi }=\ord _{\xi }^{(d)}(X)\mbox{.}$$
\end{Thm}
Before giving the proof of the Theorem  (which is detailed in  \ref{proof_theorem} below) let us make a few remarks about the result. 

\begin{Rem}  When $k$ is a perfect field of positive characteristic, by Theorem \ref{Th:PresFinita} there is a local presentation of the multiplicity function in an (\'etale) neighborhood  of a point $\xi\in \Mm$. This is given by some Rees algebra $\Gn$ defined in some smooth scheme $\Vn$ over $k$.    From this information, the invariant $\ord _{\xi }^{(d)}(X)$ is defined (see \ref{hironaka_order_d}).  However, as indicated in \ref{hironaka_order_d}, this number  does not suffice to construct a simplification of the multiplicity of $X$: it is just too coarse. 
From this perspective, the output of Theorem \ref{Main_Theorem} gives us: 
\begin{enumerate}
\item[(i)] A clue about the geometrical (intrinsic) meaning of the rational number $\ord _{\xi }^{(d)}(X)$ (see Remark \ref{significado}) and at the same time a possible explanation  about why this number shows up when trying to find a resolution. Example \ref{ejemplo_car} illustrates this idea.
\item[(ii)]  A hint to keep looking for invariants that can help refining $\ord _{\xi }^{(d)}(X)$; maybe by looking  at suitable arcs in  ${\mathcal L}(X)_{\xi}$, or maybe one can explore the use of Nash multiplicity sequences in resolution.
\end{enumerate}   
\end{Rem}

%\vspace{0.2cm}

\begin{Parrafo} \label{proof_theorem} {\em Proof of Theorem \ref{Main_Theorem}}  First we recall the definition of Hironaka's order function in dimension $d$ at a point $\xi\in \Mm$, $\ord _{\xi }^{(d)}(X)$.  By Theorem \ref{Th:PresFinita}, in some (\'etale) neighborhood of $\xi$ there is  an embedding of $X$ in an $n$-dimensional smooth scheme $V^{(n)}$ together with a (differential) Rees algebra $\Gn$ that represents the maximum multiplicity in a neighborhood of $\xi$. By the arguments in \ref{eliminacion_finita},  $\tau_{\Gn, \xi}\geq (n-d)$, and we can construct a $\Gn$-admissible projection to some $d$-dimensional smooth scheme $V^{(d)}$, $$\beta: \Vn\to \Vd$$ together with an elimination algebra $\Gd\subset \Ovd[W]$. By   \ref{EliminationProperties} (1), $\Sing(\Gn)$ is homeomorphic to $\Sing(\Gd)$, and then Hironaka's order function in dimension $d$ is defined as: 
	$$\ord _{\xi }^{(d)}(X)= \ord^{(d)}_{\xi}(\Gn)=  \ord_{\beta(\xi)}\Gd.$$
As indicated in \ref{hironaka_order_d}, this number does not depend on the choice of the $\Gn$-admissible projection, and it neither does on the choice of the embedding $X\subset \Vn$ or the Rees algebra $\Gn$. 

Thus, to show the inequality  
\begin{equation} \label{eq:Desigualdad}
\ord _{\xi }^{(d)}(X)\leq  \inf \Phi _{X,\xi }, 
\end{equation}
(see Definition \ref{def:r}) we will choose a suitable local presentation of the multiplicity and a particular smooth projection to a $d$-dimensional smooth scheme.

Since the statement of the Theorem is local, we may assume that  $X=\Spec(B)$ is  an affine algebraic variety. Then, using the arguments in \ref{PresFinita},   at a suitable (\'etale) neighborhood of $\xi$ there is an embedding in some smooth $n$-dimensional scheme $\Vn=\text{Spec}(S[x_1,\ldots, x_{n-d}])$ together with a finite morphism from $X$ to a  regular $\Vd=\text{Spec}(S) $ and a  local presentation by the differential  Rees algebra $\Gn$ generated by elements $f_1W^{m_1}, \ldots, f_{n-d}W^{m_{n-d}}$ as in \ref{PresFinita}. Recall that, in addition, $\langle f_1,\ldots, f_{n-d}\rangle\subset {\mathcal I}(X)$, the defining ideal of $X$ in $\Vn$.  
  So we have the following commutative diagram:  
\begin{equation}\label{diagrama_presentacion_1}
\xymatrix{
 S[x_1,\ldots ,x_{n-d}] \ar[r] &     S[x_1,\ldots ,x_{n-d}]  /\langle f_1,\ldots, f_{n-d}\rangle \ar[r] & B \\
 S  \ar[u]^{\beta^*}   \ar[urr] & & 
}
\end{equation}
As indicated in \ref{eliminacion_finita}, the morphism $\beta: \Vn\to \Vd$ is $\Gn$-admissible and hence it defines an elimination algebra 
 $\G^{(d)}=\Gn\cap S[W]$. 
 Now,  
$$\ord _{\xi }^{(d)}(X)=\ord_{\beta(\xi)}\Gd.$$
 By Definition \ref{def:r},
$$\Phi _{X,\xi }=\left\{ \overline{r}_{X,\varphi } \right\} _{\varphi }\subset \mathbb{Q}_{\geq 1}\mbox{,}$$
where for a given arc $\varphi $   in ${\mathcal L}(X)$ with center $\xi$ 
\begin{equation*}
\bar{r}_{X,\varphi }=\frac{\mathrm{ord}_{t}(\varphi (\Gn ))}{\nu_t(\varphi )}\in \mathbb{Q}
\end{equation*}
(see (\ref{r_escrito})). Recall that   if $\varphi: \text{Spec}(K[|t|])\to X$ for some $K\supset k$,  then  $\mathrm{ord}_{t} (\varphi (\Gn ))$ denotes the order of the $K[|t|]$-Rees  algebra at $\langle t\rangle$ while 
$\nu_t(\varphi)$   denotes the usual order of the ideal generated by $\varphi({\mathfrak m}_{\xi})$ at the (regular) local ring $K[|t|]$.   On the other hand, observe that any arc $\varphi$ as before, induces an arc in $\Vn$, which we also denote by $\varphi$, and an arc $\varphi^{(d)}$ in $\Vd$ centered at $\beta(\xi)$ together with a commutative diagram:
$$\xymatrix{R:=S[x_1,\ldots, x_{n-d}]  \ar[r] & B \ar[r]^{\varphi}& K[|t|]\\
 & \ar[ul] S \ar[u]\ar[ur]_{\varphi^{(d)}} & }
$$

Now, since $\Gd\subset \Gn_{|_B}$ is a finite extension of $B$-Rees algebras (see \ref{eliminacion_finita}),  one has by (\ref{entero_arcos}),  
$$\ord_t\varphi(\Gn_{|_B})=\ord_t\varphi(\Gd),$$ 
(note  that $\ord_t\varphi(\Gn )=\ord_t\varphi(\Gn_{|_B})$).   As ${\mathfrak m}_{\xi^{(d)}}B_{{\mathfrak m}_{\xi}}$ is a reduction of ${\mathfrak m}_{\xi}$ (see \ref{PresFinita}), one has   $\nu_t (\varphi ({\mathfrak m}_{\xi}))=\nu_t (\varphi^{(d)}({\mathfrak m}_{\beta(\xi)}))$.   Hence, 
\begin{equation}
\overline{r}_{X,\varphi }=\frac{\ord _{\xi }\varphi (\Gn )}{\nu_t(\varphi)}=\frac{\ord _{\xi }\varphi^{(d)} (\G^{(d)})}{\nu_t (\varphi^{(d)})}. 
\end{equation}
Finally, in general 
 $\ord_t  \varphi^{(d)} (\G^{(d)})\geq \nu_t (\varphi^{(d)})  \cdot \ord_{\xi^{(d)}}(\G^{(d)})$.  Thus
$$\overline{r}_{X,\varphi } =\frac{\ord_t\varphi^{(d)} (\G^{(d)})}{\nu_t(\varphi^{(d)})} \geq \ord_{\beta(\xi)}(\G^{(d)})=\ord _{\xi }^{(d)}(X).$$

To conclude the proof it suffices to show that there is an arc $\varphi\in {\mathcal L}(X)$ for which 
\begin{equation}
\label{igualdad_arcos}
\frac{\ord_t\varphi^{(d)} (\G^{(d)})}{\nu_t(\varphi^{(d)})}=  \ord_{\beta(\xi)}(\G^{(d)})=\ord _{\xi }^{(d)}(X).
\end{equation}

Let us first choose an arc $\tilde{\varphi }^{(d)}$ in $V^{(d)}$ centered at $\beta(\xi)$ for which 
$$\frac{\ord_t\tilde{\varphi}^{(d)} (\G^{(d)})}{\nu_t(\tilde{\varphi}^{(d)})}=  \ord_{\beta(\xi)}(\G^{(d)}).$$
Note that such an arc always exists: first sellect  some element $gW^l \in \G ^{(d)}$ such that 
\begin{equation}
\label{order_G}
\ord_{\beta(\xi)} (\G ^{(d)})=\frac{\nu_{\beta(\xi)} (g)}{l}=\frac{s}{l},
\end{equation}
where $\nu_{\beta(\xi)} (g)$ is the usual order at ${\mathcal O}_{\Vd, {\beta(\xi)}}$. 
And then define an arc $\tilde{\varphi }^{(d)}$  in $\Vd$, by first fixing  a regular system of parameters,    $y_1,\ldots, y_d\in   {\mathcal O}_{\Vd,{\beta(\xi)}}$,   and then passing to the completion: 
$$\begin{array}{rcccl} 
{\mathcal O}_{\Vd, {\mathfrak m}_{\beta(\xi)}}  &  \to  &   \widehat{\mathcal O}_{\Vd, {\mathfrak m}_{\beta(\xi)}}\simeq k'[|Y_1,\ldots, Y_d|] & \to & k'[|t|] \\
y_i  & \mapsto & Y_i & \mapsto & u_i t^{\alpha} 
\end{array}$$  
where     $\alpha \in \mathbb{Z}_{>0}$ and $u_1,\ldots ,u_d$ are suitably chosen  units  in $k'[[t]]$\footnote{Here  $k'$, the residue field at $\beta(\xi)$ may have to be replaced by an \'etale field extension so that condition (\ref{initial_G}) holds.} such that  
\begin{equation}
\label{initial_G}
(\mathrm{in}_{\beta(\xi)}(g))(u_1,\ldots ,u_d)\neq 0 
\end{equation}
where $\mathrm{in}_{\xi^{(d)}}(g)$ denotes the {\em initial part of $g$ at $\xi$}\footnote{If $\nu_{\beta(\xi)}(g)=s$,  then $\mathrm{in}_{\xi^{(d)}}(g)$  denotes  the class of $g$ at $\mathfrak{m}_{\beta(\xi)}^s/\mathfrak{m}_{\beta(\xi)}^{s+1}$;   therefore $\mathrm{in}_{\xi }(g)\in \mathrm{Gr}_{{\mathfrak m}_{\beta(\xi)}}({S_{\mathfrak {m}_{\beta(\xi)}}})\cong k'[Z_1,\ldots ,Z_d]$ is a homogeneous polynomial of degree $s$.}. From the way in which $\tilde{\varphi}$ is defined, 
\begin{equation}
\ord_{\beta(\xi)}(\G^{(d)})\leq \frac{\ord_t\tilde{\varphi}^{(d)} (\G^{(d)})}{\nu_t(\tilde{\varphi}^{(d)})}\leq  \frac{\nu_t\tilde{\varphi}^{(d)} (g)/l}{\nu_t(\tilde{\varphi}^{(d)})}=\frac{\alpha\cdot s/l}{\alpha} =\frac{s}{l}=\ord_{\beta(\xi)}(\G^{(d)}).
\end{equation}

 From this  arc $\tilde{\varphi }^{(d)}$, we will construct an arc $\varphi$  in $X$ centered at $\xi $ whose projection to an arc $\varphi^{(d)}$ in $V^{(d)}$ will give the equality in (\ref{igualdad_arcos}).

The arc $\tilde{\varphi }^{(d)}$ is fat in a closed subvariety $Y\subset V^{(d)}$, which is the closure of its generic point in $\Vd$. Denote by $I(Y)\subset S$ the  ideal defining $Y$ as a subset of $V^{(d)}$, and define $S'=S/I(Y)$. Let $J\subset B$ be some prime ideal dominating $I(Y)$.  Then we have a commutative diagram of finite vertical morphisms, 
\begin{equation*}
\xymatrix{
	B \ar[r] & B'=B/J & \\
	S \ar[u] \ar[r] & S' \ar[u] \ar[r]^{\tilde{\varphi }^{(d)}} & k'[[t]].
}
\end{equation*}
Now, $\tilde{\varphi }^{(d)}$ defines a discrete  valuation $\tilde{v}$ on $K(S')$, the quotient field of $S'$, whose  valuation ring 
$\mathcal{O}_{\tilde{v}}$ contains $S'$.  If $K(B')$ is the quotient field of $B'$, then the extension $K(S') \subset K(B')$ is finite, and $\mathcal{O}_{\tilde{v}}$ is dominated by a finite number of discrete valuation rings in $K(B')$,  all of them dominating $B'$. Denote by   ${\mathcal O}_v$ one of these (discrete) valuation rings, and by $v$ the corresponding valuation. Then  the inclusions, 
$$S'\subset B'\subset {\mathcal O}_v\to \widehat{{\mathcal O}_v}\simeq K_v[|t|],$$ 
define  an arc $\varphi: S\to K_v[|t|]$    that we claim gives the equality in (\ref{igualdad_arcos}). To prove the claim, let   
  $gW^l\in \Gd$ be as in (\ref{order_G}) satisfying (\ref{initial_G}). Now, if the ramification index of $\tilde{v}$ in ${\mathcal O}_v$ is $N\in {\mathbb Z}_{>0}$, then, 
$$\ord_{\beta(\xi)}\Gd\leq \frac{\ord_t\varphi(\Gd)}{\nu_t(\varphi)} \leq   \frac{\nu_t\varphi(g)/l}{\nu_t(\varphi)}=
\frac{v(g{\mathcal O}_v)/l}{v(\mathfrak{m}_{\beta(\xi)}\mathcal{O}_{{v}})}=$$
$$=\frac{N\cdot \tilde{v}(g\mathcal{O}_{\tilde{v}})/l}{N\cdot \tilde{v}(\mathfrak{m}_{\beta(\xi)}\mathcal{O}_{\tilde{v}})}=\frac{\ord_t\tilde{\varphi}^{(d)} (\Gd)}{\nu_t(\tilde{\varphi}^{(d)})}=\ord_{\beta(\xi)}\Gd\mbox{.} $$ \qed

\end{Parrafo}

\begin{Ex}\label{ejemplo_car}
Let $k$ be a perfect field, let $R=k[x,y]$, let $B=k[x,y]/\langle y^2-x^3\rangle$  and let $X=\text{Spec}(B)$. Then $\mm=2$ and $\Mm=\{\xi=(0,0)\}$. 
 Up to integral closure, the differential $R$-Rees algebra representing $\Mm$ is $\G^{(2)}=R[yW, x^2W,  x^3W^2]$ 
if the characteristic is different from 2, and ${\mathcal H}^{(2)}=R[x^2W, (y^2-x^3)W^2]$ if the  characteristic is 2. In both cases the natural inclusion $S=k[x]\subset R=k[x,y]$ is admissible. The elimination algebra for $\G^{(2)}$ is $\G^{(1)}=S[x^2W,  x^3W^2]$, and the one for ${\mathcal H}^{(2)}$ is  ${\mathcal H}^{(1)}=S[x^2W]$. Thus we have two different values for Hironaka's order function in dimension $1=\text{dim}_k(X)$, depending on the characteristic: 
$$\text{ord}^{(1)}_{\xi}(X) =\left\{\begin{array}{ll}
\frac{3}{2} & \text{ if } \text{char}(k)\neq 2; \\
   &   \\
2  & \text{ if } \text{char}(k)= 2. \end{array} \right.
$$
Lowering $\mm$ below 2 takes just one blow up at $\xi$: this is what it takes to resolve both $\G^{(2)}$  and ${\mathcal H}^{(2)}$ (here we forget about the normal crossing conditions because this is not relevant to the example).

\

 When the characteristic is zero, resolving $\G^{(2)}$ is equivalent to resolving $\G^{(1)}$ and the fact that one single blow up is enough is reflected in the value $3/2$: $\G^{(1)}$ is resolved in one step. 
 
 \

When the characteristic is 2, ${\mathcal H}^{(1)}$ somehow {\em exagerates} the image of the singular locus of  ${\mathcal H}^{(2)}$: it takes two blow ups to resolve  ${\mathcal H}^{(1)}$ while  ${\mathcal H}^{(2)}$ is resolved by one. And yet, there is no other $S$-Rees algebra that approximates the image of the singular locus of  ${\mathcal H}^{(2)}$ than  ${\mathcal H}^{(1)}$. Thus: why the value $\text{ord}^{(1)}_{\xi}(X)=2$?

\

Let us  look at the problem from the point of view of the arcs in $X$ with center $\xi$, and consider: 
$$\begin{array}{rrcl}
\varphi: & k[x,y] & \longrightarrow & k[|t|]\\
  & x & \mapsto & t^3\\
   & y & \mapsto & t^2.
   \end{array}$$ 

Now, if we compute the Nash multiplicity sequence of $\varphi$ and the persistance (normalized), we obtain: 
$$\begin{array}{lll}
2=m_0=m_1=m_2>m_3=1;  & \overline{\rho}_{X,\varphi}=\frac{3}{2} & \text{ if } \text{char}(k)\neq 2; \\
&   \\
2=m_0=m_1=m_2=m_3>m_4=1; \    & \overline{\rho}_{X,\varphi}= 2 & \text{ if } \text{char}(k)= 2. \end{array} 
$$
Thus, in the characteristic 2 case, it takes longer to separate the graph of the arc from the maximum multiplicity locus of $X\times {\mathbb A}^1_k$ and the order of ${\mathcal H}^{(1)}$ at the origin is reflecting this fact: this order  cannot take a value   below 2. 
\end{Ex}


\begin{thebibliography}{10}

\bibitem{COA}
C.~Abad, A.~Bravo, and O.E. Villamayor~U.
\newblock Finite morphisms and simultaneous reduction of the multiplicity.
\newblock {\em Preprint, arXiv:1710.01805v2 [math.AG]}, 2018.

\bibitem{Abhy1}
S.~Abhyankar.
\newblock Corrections to ``{L}ocal uniformization on algebraic surfaces over
  ground fields of characteristic {$p\not=0$}''.
\newblock {\em Ann. of Math. (2)}, 78:202--203, 1963.

\bibitem{Abhy2}
S.~Abhyankar.
\newblock Uniformization in {$p$}-cyclic extensions of algebraic surfaces over
  ground fields of characteristic {$p$}.
\newblock {\em Math. Ann.}, 153:81--96, 1964.

\bibitem{B}
A.~Benito.
\newblock The {$\tau$}-invariant and elimination.
\newblock {\em J. Algebra}, 324(8):1903--1920, 2010.

\bibitem{Benito_V}
A.~Benito and O.~E. Villamayor~U.
\newblock Techniques for the study of singularities with applications to
  resolution of 2-dimensional schemes.
\newblock {\em Math. Ann.}, 353(3):1037--1068, 2012.

\bibitem{Bennett}
B.~M. Bennett.
\newblock On the characteristic functions of a local ring.
\newblock {\em Ann. of Math. (2)}, 91:25--87, 1970.

\bibitem{Bhatt}
Bhargav Bhatt.
\newblock Algebraization and {T}annaka duality.
\newblock {\em Camb. J. Math.}, 4(4):403--461, 2016.

\bibitem{B-M}
E.~Bierstone and P.~D. Milman.
\newblock Canonical desingularization in characteristic zero by blowing up the
  maximum strata of a local invariant.
\newblock {\em Invent. Math.}, 128(2):207--302, 1997.

\bibitem{Br_E_P-E}
A.~Bravo, S.~Encinas, and B.~Pascual-Escudero.
\newblock Nash multiplicities and resolution invariants.
\newblock {\em Collectanea Mathematica}, 68(2):175--217, 2017.

\bibitem{Br_G-E_V}
A.~Bravo, M.~L. Garcia-Escamilla, and O.~E. Villamayor~U.
\newblock On {R}ees algebras and invariants for singularities over perfect
  fields.
\newblock {\em Indiana Univ. Math. J.}, 61(3):1201--1251, 2012.

\bibitem{Br_V}
A.~Bravo and O.~Villamayor~U.
\newblock Singularities in positive characteristic, stratification and
  simplification of the singular locus.
\newblock {\em Adv. Math.}, 224(4):1349--1418, 2010.

\bibitem{Br_V2}
A.~Bravo and O.~E. Villamayor~U.
\newblock On the behavior of the multiplicity on schemes: stratification and
  blow ups.
\newblock In {\em The resolution of singular algebraic varieties}, pages
  81--207. Amer. Math. Soc., Providence, RI, 2014.

\bibitem{CJS_prep}
V.~Cossart, U.~Jannsen, and S.~Saito.
\newblock Canonical embedded and non-embedded resolution of singularities for
  excellent two-dimensional schemes.
\newblock {\em arXiv:0905.2191}.

\bibitem{Cos_Pilt1}
V.~Cossart and O.~Piltant.
\newblock Resolution of singularities of threefolds in positive characteristic.
  {I}. {R}eduction to local uniformization on {A}rtin-{S}chreier and purely
  inseparable coverings.
\newblock {\em J. Algebra}, 320(3):1051--1082, 2008.

\bibitem{Cos_Pilt2}
V.~Cossart and O.~Piltant.
\newblock Resolution of singularities of threefolds in positive characteristic.
  {II}.
\newblock {\em J. Algebra}, 321(7):1836--1976, 2009.

\bibitem{Dade}
E.~D. Dade.
\newblock {\em Multiplicity and monoidal transformations}.
\newblock Princeton University, 1960.
\newblock Thesis (Ph.D.).

\bibitem{deF_Doc}
T.~de~Fernex and R.~Docampo.
\newblock Terminal valuations and the {N}ash problem.
\newblock {\em Invent. Math.}, 203(1):303--331, 2016.

\bibitem{D_L}
J.~Denef and F.~Loeser.
\newblock Geometry on arc spaces of algebraic varieties.
\newblock In {\em European {C}ongress of {M}athematics, {V}ol. {I}
  ({B}arcelona, 2000)}, volume 201 of {\em Progr. Math.}, pages 327--348.
  Birkh\"auser, Basel, 2001.

\bibitem{E_M04}
L.~Ein and M.~Musta{\c{t}}{\v{a}}.
\newblock Inversion of adjunction for local complete intersection varieties.
\newblock {\em Amer. J. Math.}, 126(6):1355--1365, 2004.

\bibitem{E_M09}
L.~Ein and M.~Musta{\c{t}}{\u{a}}.
\newblock Jet schemes and singularities.
\newblock In {\em Algebraic geometry---{S}eattle 2005. {P}art 2}, volume~80 of
  {\em Proc. Sympos. Pure Math.}, pages 505--546. Amer. Math. Soc., Providence,
  RI, 2009.

\bibitem{E_M_Y}
L.~Ein, M.~Musta{\c{t}}{\u{a}}, and T.~Yasuda.
\newblock Jet schemes, log discrepancies and inversion of adjunction.
\newblock {\em Invent. Math.}, 153(3):519--535, 2003.

\bibitem{E_Hau}
S.~Encinas and H.~Hauser.
\newblock Strong resolution of singularities in characteristic zero.
\newblock {\em Comment. Math. Helv.}, 77(4):821--845, 2002.

\bibitem{E_V_}
S.~Encinas and O.~Villamayor.
\newblock Good points and constructive resolution of singularities.
\newblock {\em Acta Math.}, 181(1):109--158, 1998.

\bibitem{E_V97}
S.~Encinas and O.~Villamayor.
\newblock A course on constructive desingularization and equivariance.
\newblock In {\em Resolution of singularities ({O}bergurgl, 1997)}, volume 181
  of {\em Progr. Math.}, pages 147--227. Birkh\"auser, Basel, 2000.

\bibitem{E_V}
S.~Encinas and O.~Villamayor.
\newblock Rees algebras and resolution of singularities.
\newblock In {\em Proceedings of the {XVI}th {L}atin {A}merican {A}lgebra
  {C}olloquium ({S}panish)}, Bibl. Rev. Mat. Iberoamericana, pages 63--85. Rev.
  Mat. Iberoamericana, Madrid, 2007.

\bibitem{Giraud}
J.~Giraud.
\newblock Contact maximal en caract\'eristique positive.
\newblock {\em Ann. Sci. \'Ecole Norm. Sup. (4)}, 8(2):201--234, 1975.

\bibitem{Hickel05}
M.~Hickel.
\newblock Sur quelques aspects de la g\'eom\'etrie de l'espace des arcs
  trac\'es sur un espace analytique.
\newblock {\em Ann. Fac. Sci. Toulouse Math. (6)}, 14(1):1--50, 2005.

\bibitem{Hir}
H.~Hironaka.
\newblock Resolution of singularities of an algebraic variety over a field of
  characteristic zero. {I}, {II}.
\newblock {\em Ann. of Math. (2) {79} (1964), 109--203; ibid. (2)},
  79:205--326, 1964.

\bibitem{Hir1}
H.~Hironaka.
\newblock Idealistic exponents of singularity.
\newblock In {\em Algebraic geometry ({J}. {J}. {S}ylvester {S}ympos., {J}ohns
  {H}opkins {U}niv., {B}altimore, {M}d., 1976)}, pages 52--125. Johns Hopkins
  Univ. Press, Baltimore, Md., 1977.

\bibitem{Hir2}
Heisuke Hironaka.
\newblock Certain numerical characters of singularities.
\newblock {\em J. Math. Kyoto Univ.}, 10:151--187, 1970.

\bibitem{I2}
S.~Ishii.
\newblock Smoothness and jet schemes.
\newblock In {\em Singularities---{N}iigata--{T}oyama 2007}, volume~56 of {\em
  Adv. Stud. Pure Math.}, pages 187--199. Math. Soc. Japan, Tokyo, 2009.

\bibitem{I}
S.~Ishii.
\newblock Geometric properties of jet schemes.
\newblock {\em Comm. Algebra}, 39(5):1872--1882, 2011.

\bibitem{I_K}
S.~Ishii and J.~Koll\'ar.
\newblock The {N}ash problem on arc families of singularities.
\newblock {\em Duke Math. J.}, 120(3):601--620, 2003.

\bibitem{MR3617780}
Hiraku Kawanoue and Kenji Matsuki.
\newblock Resolution of singularities of an idealistic filtration in dimension
  3 after {B}enito-{V}illamayor.
\newblock In {\em Minimal models and extremal rays ({K}yoto, 2011)}, volume~70
  of {\em Adv. Stud. Pure Math.}, pages 115--214. Math. Soc. Japan, [Tokyo],
  2016.

\bibitem{L-J}
M.~Lejeune-Jalabert.
\newblock Courbes trac\'ees sur un germe d'hypersurface.
\newblock {\em Amer. J. Math.}, 112(4):525--568, 1990.

\bibitem{L-J_Mou_Re}
M.~Lejeune-Jalabert, H.~Mourtada, and A.~Reguera.
\newblock Jet schemes and minimal embedded desingularization of plane branches.
\newblock {\em Rev. R. Acad. Cienc. Exactas F\'\i s. Nat. Ser. A Math. RACSAM},
  107(1):145--157, 2013.

\bibitem{L-J_Re1}
M.~Lejeune-Jalabert and A.~J. Reguera-L\'opez.
\newblock Arcs and wedges on sandwiched surface singularities.
\newblock {\em Amer. J. Math.}, 121(6):1191--1213, 1999.

\bibitem{Lipman3}
J.~Lipman.
\newblock Desingularization of two-dimensional schemes.
\newblock {\em Ann. Math. (2)}, 107(1):151--207, 1978.

\bibitem{Mou4}
H.~Mourtada.
\newblock Jet schemes of rational double point singularities.
\newblock In {\em Valuation theory in interaction}, EMS Ser. Congr. Rep., pages
  373--388. Eur. Math. Soc., Z\"urich, 2014.

\bibitem{Mus1}
M.~Musta{\c{t}}{\u{a}}.
\newblock Jet schemes of locally complete intersection canonical singularities.
\newblock {\em Invent. Math.}, 145(3):397--424, 2001.
\newblock With an appendix by David Eisenbud and Edward Frenkel.

\bibitem{Mus2}
M.~Musta{\c{t}}{\v{a}}.
\newblock Singularities of pairs via jet schemes.
\newblock {\em J. Amer. Math. Soc.}, 15(3):599--615 (electronic), 2002.

\bibitem{Nash}
J.~F. Nash, Jr.
\newblock Arc structure of singularities.
\newblock {\em Duke Math. J.}, 81(1):31--38 (1996), 1995.
\newblock A celebration of John F. Nash, Jr.

\bibitem{Veys}
W.~Veys.
\newblock Arc spaces, motivic integration and stringy invariants.
\newblock In {\em Singularity theory and its applications}, volume~43 of {\em
  Adv. Stud. Pure Math.}, pages 529--572. Math. Soc. Japan, Tokyo, 2006.

\bibitem{V1}
O.~Villamayor.
\newblock Constructiveness of {H}ironaka's resolution.
\newblock {\em Ann. Sci. \'Ecole Norm. Sup. (4)}, 22(1):1--32, 1989.

\bibitem{V2}
O.~Villamayor.
\newblock Patching local uniformizations.
\newblock {\em Ann. Sci. \'Ecole Norm. Sup. (4)}, 25(6):629--677, 1992.

\bibitem{V00}
O.~E. Villamayor.
\newblock Tschirnhausen transformations revisited and the multiplicity of the
  embedded hypersurface.
\newblock {\em Bol. Acad. Nac. Cienc. (C\'ordoba)}, 65:233--243, 2000.
\newblock Colloquium on Homology and Representation Theory (Spanish)
  (Vaquer{\'{\i}}as, 1998).

\bibitem{V07}
O.~Villamayor~U.
\newblock Hypersurface singularities in positive characteristic.
\newblock {\em Adv. Math.}, 213(2):687--733, 2007.

\bibitem{V3}
O.~Villamayor~U.
\newblock Rees algebras on smooth schemes: integral closure and higher
  differential operator.
\newblock {\em Rev. Mat. Iberoam.}, 24(1):213--242, 2008.

\bibitem{V}
O.~E. Villamayor~U.
\newblock Equimultiplicity, algebraic elimination, and blowing-up.
\newblock {\em Adv. Math.}, 262:313--369, 2014.

\bibitem{Vojta}
P.~Vojta.
\newblock Jets via {H}asse-{S}chmidt derivations.
\newblock In {\em Diophantine geometry}, volume~4 of {\em CRM Series}, pages
  335--361. Ed. Norm., Pisa, 2007.

\bibitem{Z-SII}
O.~Zariski and P.~Samuel.
\newblock {\em Commutative algebra. {V}ol. {II}}.
\newblock The University Series in Higher Mathematics. D. Van Nostrand Co.,
  Inc., Princeton, N. J.-Toronto-London-New York, 1960.

\end{thebibliography}
\end{document}